\documentclass[12pt]{article}
\usepackage{amssymb,amsthm,bbm,amsmath,array}
\usepackage[small]{titlesec} 
\titleformat{\section}{\Large\bfseries}{\thesection}{1.5pc}{} \titlespacing{\section}{0cm}{3pc}{2pc}
\titleformat{\subsection}{\it}{\thesubsection.}{0.5pc}{} \titlespacing{\subsection}{0cm}{1.5pc}{3mm}
\renewcommand{\thesection}{\bf \S\arabic{section}.}

\makeatletter
\def\@citex[#1]#2{\if@filesw\immediate\write\@auxout{\string\citation{#2}}\fi
  \def\@citea{}\@cite{\@for\@citeb:=#2\do
    {\@citea\def\@citea{\@citesep}\@ifundefined
       {b@\@citeb}{{\bf ?}\@warning
       {Citation `\@citeb' on page \thepage \space undefined}}%
{\csname b@\@citeb\endcsname}}}{#1}}
\def\@citesep{; }

\makeatother

\newtheoremstyle{HU}{}{}{\itshape}{}{\bf}{}{.6em}{}
\theoremstyle{HU}
\newtheorem{theorem}{Theorem}[section]

\newtheorem{lemma}[theorem]{Lemma}
\newtheorem{prop}[theorem]{Proposition}
\newtheorem{defn}[theorem]{Definition}

\newtheoremstyle{Hremark}{}{}{}{}{\bf}{}{.6em}{}
\theoremstyle{Hremark}

\newtheorem{other}{}
\newenvironment{idef}[1]{\renewcommand{\theother}{#1.}\begin{other}}{\end{other}}

\allowdisplaybreaks[1]  

\numberwithin{figure}{section}

\def\fn#1{\operatorname{#1}} 
\def\bm#1{\mathbbm{#1}}

\def\dW{\mathop{\hbox{\hbox to 1eM{$\underline{W\!\!}$}}}}

\title{Noether's Problem for Groups of Order 243}
\author{Huah Chu$^{(1)}$, Akinari Hoshi$^{(2)}$, Shou-Jen Hu$^{(3)}$, Ming-chang Kang$^{(4)}$ \\[3mm]
\begin{minipage}{16cm} \begin{description} \itemsep=-1pt
\item[] $^{(1)}$Department of Mathematics, National Taiwan
University, Taipei,\\ Taiwan; hchu@math.ntu.edu.tw
\item[]$^{(2)}$Department of Mathematics, Niigata University, Niigata, Japan;\\
hoshi@math.sc.niigata-u.ac.jp
\item[]$^{(3)}$Department of Mathematics, Tamkang University, Taipei, Taiwan;\\
sjhu@math.tku.edu.tw \item[] $^{(4)}$Department of Mathematics,
National Taiwan University, Taipei,\\ Taiwan; kang@math.ntu.edu.tw
\end{description} \end{minipage}}
\date{}

\begin{document}

\maketitle \footnote{\hspace*{-7.5mm} 2010 Mathematics Subject
Classification.
Primary 13A50, 11R32, 14E08, Secondary 12F12. 
The second-named author was partially supported by JSPS KAKENHI,
Grant Number 25400027.}

\begin{abstract}{\noindent Abstract.}
Let $k$ be any field, $G$ be a finite group. Let $G$ act on the
rational function field $k(x_g:g\in G)$ by $k$-automorphisms
defined by $h\cdot x_g=x_{hg}$ for any $g,h\in G$. Denote by
$k(G)=k(x_g:g\in G)^G$ the fixed field. Noether's problem asks,
under what situations, the fixed field $k(G)$ will be rational (=
purely transcendental) over $k$. According to the data base of GAP
there are $10$ isoclinism families for groups of order $243$.  It
is known that there are precisely $3$ groups $G$ of order $243$
(they consist of the isoclinism family $\Phi_{10}$) such that the
unramified Brauer group of $\bm{C}(G)$ over $\bm{C}$ is
non-trivial. Thus $\bm{C}(G)$ is not rational over $\bm{C}$. We
will prove that, if $\zeta_9 \in k$, then $k(G)$ is rational over
$k$ for groups of order $243$ other than these $3$ groups, except
possibly for groups belonging to the isoclinism family $\Phi_7$.
\end{abstract}

\newpage
\section{Introduction}

Let $k$ be a field, and $L$ be a finitely generated field
extension of $k$. $L$ is called $k$-rational (or rational over
$k$) if $L$ is purely transcendental over $k$, i.e. $L$ is
isomorphic to some rational function field over $k$. $L$ is called
stably $k$-rational if $L(y_1,\ldots,y_m)$ is $k$-rational for
some $y_1,\ldots,y_m$ which are algebraically independent over
$L$. $L$ is called $k$-unirational if $L$ is $k$-isomorphic to a
subfield of some $k$-rational field extension of $k$. It is easy
to see that ``$k$-rational" $\Rightarrow$ ``stably $k$-rational"
$\Rightarrow$ ``$k$-unirational".

A classical question, the L\"uroth problem by some people, asks
whether a $k$-unirational field $L$ is necessarily $k$-rational.
For a survey of the question, see \cite{MT} and \cite{CTS}.

Noether's problem is a special case of the above L\"uroth problem.

Let $k$ be a field and $G$ be a finite group. Let $G$ act on the
rational function field $k(x_g:g\in G)$ by $k$-automorphisms
defined by $h\cdot x_g=x_{hg}$ for any $g,h\in G$. Denote by
$k(G)$ the fixed subfield, i.e. $k(G)=k(x_g:g\in G)^G$. Noether's
problem asks, under what situation, the field $k(G)$ is
$k$-rational.

\begin{theorem}[Fischer {\cite{Fi}, see also \cite[Theorem 6.1]{Sw2}}] \label{t2.2}
Let $G$ be a finite abelian group of exponent $e$,
$k$ be a field containing a primitive $e$-th root of unity.
Then $k(G)$ is $k$-rational.
\end{theorem}

\begin{theorem}[Kuniyoshi, Gasch\"utz \cite{Ku1}, \cite{Ku2}, \cite{Ku3}, \cite{Ga}] \label{t2.1}
Let $k$ be a field with $\fn{char}k=p>0$, $G$ be a finite
$p$-group. Then $k(G)$ is $k$-rational.
\end{theorem}

Noether's problem is related to the inverse Galois problem,
to the existence of generic $G$-Galois extensions over $k$, and
to the existence of versal $G$-torsors over $k$-rational field extensions
\cite{Sw2}, \cite{Sa1}, \cite[Section 33.1, page 86]{GMS}.

The first counter-example to Noether's problem was constructed by
Swan \cite{Sw1}: $\bm{Q}(C_p)$ is not $\bm{Q}$-rational if $p=47$,
$113$ or $233$ etc.\ where $C_p$ is the cyclic group of order $p$.
Noether's problem for finite abelian groups was studied
extensively by Swan, Voskresenskii, Endo and Miyata, Lenstra, etc.
For details, see Swan's survey paper \cite{Sw2}.

On the other hand, the results of Noether's problem for non-abelian
groups are rather scarce. First of all, recall a notion of retract
$k$-rationality introduced by Saltman (see \cite{Sa3} or
\cite{Ka3}). It is known from the definition of retract
$k$-rationality that, if $k$ is an infinite field, then ``stably
$k$-rational" $\Rightarrow$ ``retract $k$-rational" $\Rightarrow$
``$k$-unirational". It follows that, if $k(G)$ is not retract
$k$-rational, then it is not $k$-rational.

In \cite{Sa2}, Saltman defines $\fn{Br}_{v,k}(k(G))$, the unramified
Brauer group of $k(G)$ over $k$. It is known that, if $k(G)$ is
retract $k$-rational, then the natural map $Br (k) \rightarrow
\fn{Br}_{v,k}(k(G))$ is an isomorphism; in particular, if $k$ is
algebraically closed, then $\fn{Br}_{v,k}(k(G))$ $=$ $0$. Thus the crucial
point in \cite{Sa2} is to construct a $p$-group $G$ with
$\fn{Br}_{v,k}(k(G))\ne 0$.

\begin{theorem}[Saltman \cite{Sa2}] \label{t1.1}
Let $p$ be any prime number, $k$ be any infinite field with
$\fn{char}k\ne p$.
Then there exists a meta-abelian $p$-group $G$ of order $p^9$
such that $k(G)$ is not retract $k$-rational. It follows that
$k(G)$ {\rm{(}}in particular, $\bm{C}(G)${\rm{)}} is not $k$-rational.
\end{theorem}

Bogomolov gives a formula (\cite[Theorem 3.1]{Bo})
for computing the unramified Brauer
group and he is able to improve the bound of the group order to
$p^6$.

\begin{theorem}[{Bogomolov \cite[Lemma 5.6]{Bo}}] \label{t1.2}
There exists a $p$-group $G$ of order $p^6$ such that $\fn{Br}_{v,{\bm
C}}(\bm C(G))$ $\ne 0$. It follows that $\bm{C}(G)$ is not retract
$\bm{C}$-rational; thus it is not $\bm{C}$-rational.
\end{theorem}

In \cite[Remark 1]{Bo}, Bogomolov proposes to classify all the $p$-groups $G$
of order $p^6$ with $\fn{Br}_{v,\bm{C}} (\bm{C}(G))\ne 0$.
This is done in \cite[Theorem 1.8]{CHKK} for $p=2$;
in fact, it is shown that there are precisely nine groups $G$ of
order 64, $G(64,i)$ where $i=149$, 150, 151, 170, 171, 172, 177,
178, 182 with $\fn{Br}_{v,\bm{C}}(\bm{C}(G))\ne 0$ (the notation
$G(64,i)$ denotes the $i$-th group of order 64 in the database of
GAP \cite{GAP}). Moreover, it is known that, if $G$ is a group of
order 64 with $\fn{Br}_{v,\bm{C}}(\bm{C}(G))=0$, then $\bm{C}(G)$
is $\bm{C}$-rational except possibly for five unsettled groups
$G(64,i)$ with $241\leq i\leq 245$ \cite[Theorem 1.10]{CHKK}.

The notion of the unramified Brauer group is generalized to the
unramified cohomology of degree $q$,
$H_{v,\bm{C}}^q(\bm{C}(G),\bm{Q}/\bm{Z})$ where $q\ge 2$ by
Colliot-Th\'el\`ene and Ojanguren \cite{CTO}. It is also known
that, if $\bm{C}(G)$ is retract $\bm{C}$-rational, then
$H_{v,\bm{C}}^q(\bm{C}(G),\bm{Q}/\bm{Z})=0$ for all $q \ge 2$.
Using the unramified cohomology of degree 3,
$H_{v,\bm{C}}^3(\bm{C}(G),\bm{Q}/\bm{Z})$, Peyre is able to prove
the following theorem.

\begin{theorem}[{Peyre \cite[Theorem 3]{Pe2}}] \label{t1.3}
Let $p$ be any odd prime number.
There exists a $p$-group $G$ of order $p^{12}$ such that
$\fn{Br}_{v,\bm{C}}(\bm{C}(G))=0$ and
$H_{v,\bm{C}}^3(\bm{C}(G),\bm{Q}/\bm{Z})\ne 0$.
In particular, $\bm{C}(G)$ is not stably $\bm{C}$-rational.
\end{theorem}

The triviality of the unramified Brauer group or the unramified
cohomology of higher degree is just a necessary condition of
$\bm{C}$-rationality of fields. It is unknown whether the vanishing of
all the unramified cohomologies is a sufficient condition for
$\bm{C}$-rationality.
Asok \cite{As} generalized Peyre's argument \cite{Pe1} and
established the following theorem
for a smooth projective variety $X$:

\begin{theorem}[{Asok \cite[Theorem 1, Theorem 3]{As}}]
{\rm (1)} For any $n\ge 1$, there exists a smooth projective
complex variety $X$ that is $\bm{C}$-unirational, for which
$H_{v,\bm{C}}^i(X,\mu_2^{\otimes i})=0$ for each $i<n$, yet
$H_{v,\bm{C}}^n(X,\mu_2^{\otimes n})\neq 0$, and so $X$ is not
$\mathbb{A}^1$-connected
$($nor stably $\bm{C}$-rational$)$;\\
{\rm (2)} For any prime number $l$ and any $n\geq 2$,
there exists a smooth projective rationally connected complex
variety $X$ such that
$H_{v,\bm{C}}^n(X,\mu_l^{\otimes n})\neq 0$.
In particular, $X$ is not $\mathbb{A}^1$-connected
$($nor stably $\bm{C}$-rational$)$.
\end{theorem}

We now consider $p$-groups of small order. By
Fischer's Theorem (Theorem \ref{t2.2}), if $G$ is an abelian
$p$-group and the base field $k$ contains enough roots of unity,
then $k(G)$ is $k$-rational. Hence we may focus on non-abelian
groups.

\begin{theorem}[{Chu and Kang \cite[Theorem 1.6]{CK}}] \label{t1.4}
Let $G$ be a $p$-group of order $p^3$ or $p^4$.
If $k$ is a field satisfying {\rm (i)} $\fn{char}k=p>0$, or {\rm (ii)} $\fn{char}k\ne p$ with $\zeta_e\in k$ where $e$ is the exponent of the group $G$,
then $k(G)$ is $k$-rational.
\end{theorem}

By the above Theorem \ref{t1.2} and Theorem \ref{t1.4},
it is interesting to know whether $k(G)$ is $k$-rational if $G$ is any
$p$-group of order $p^5$. Here is an answer when $p=2$.

\begin{theorem}[{Chu, Hu, Kang and Prokhorov \cite[Theorem 1.5]{CHKP}}] \label{t1.5}
Let $G$ be a group of order $32$ with exponent $e$.
If $k$ is a field satisfying {\rm (i)} $\fn{char} k=2$,
or {\rm (ii)} $\fn{char}k\ne 2$ with $\zeta_e \in k$, then $k(G)$ is $k$-rational.
\end{theorem}

What happens to groups of order $p^5$ with $p\ne 2$?

In \cite{Bo}, Bogomolov claims a property for the unramified
Brauer group.

\begin{prop}[{\cite[Lemma 4.11]{Bo}, \cite[Corollary 2.11]{BMP}}] \label{p1.6}
If $G$ is a $p$-group with $\fn{Br}_{v,\bm{C}}(\bm{C}(G))\ne 0$,
then $|G|\ge p^6$.
\end{prop}

Unfortunately, the above proposition was disproved by the following
result of Moravec.

\begin{theorem}[Moravec {\cite[Section 8]{Mo}}] \label{t1.7}
Let $G$ be a group of order $243$. Then
$\fn{Br}_{v,\bm{C}}(\bm{C}(G))\ne 0$ if and only if $G=G(243,i)$
with $28\le i\le 30$, where $G(243,i)$ is the $i$-th group among
groups of order $243$ in the GAP database.
\end{theorem}

Moravec's proof relies on computer computation.
In \cite{HoK},
a theoretic proof  of the non-vanishing of the three groups of order
$243$ is given. Recently,  Hoshi, Kang and Kunyavskii \cite{HKKu}
are able to determine which $p$-groups have
$\fn{Br}_{v,\bm{C}}(\bm{C}(G))\ne 0$ according to the isoclinism
families they belong to.

\begin{defn} \label{d1.8}
Two $p$-groups $G_1$ and $G_2$ are called isoclinic if there exist
group isomorphisms $\theta\colon G_1/Z(G_1) \to G_2/Z(G_2)$ and
$\phi\colon [G_1,G_1]\to [G_2,G_2]$ such that $\phi([g,h])$
$=[g',h']$ for any $g,h\in G_1$ with $g'\in \theta(gZ(G_1))$, $h'\in
\theta(hZ(G_1))$. {\rm {(Note that $Z(G)$ and $[G,G]$ denote the
center and the commutator subgroup of the group $G$
respectively)}}.\end{defn}

For a prime number $p$ and a fixed integer $n$, let $G_n(p)$ be the
set of all non-isomorphic groups of order $p^n$. In $G_n(p)$
consider an equivalence relation: two groups $G_1$ and $G_2$ are
equivalent if and only if they are isoclinic. Each equivalence class
of $G_n(p)$ is called an isoclinism family.
There exist ten isoclinism families $\Phi_1,\ldots,\Phi_{10}$
for groups of order $p^5$.

The main theorem in \cite {HKKu} can be stated as

\begin{theorem}[{Hoshi, Kang and Kunyavskii \cite[Theorem 1.12]{HKKu}}] \label{t1.9}
Let $p$ be any odd prime number, $G$ be a group of order $p^5$. Then
${\rm{Br}}_{v,\bm C}(\bm C(G))\ne 0$ if and only if $G$ belongs to
the isoclinism family $\Phi_{10}$. Each group $G$ in the family
$\Phi_{10}$ satisfies the condition $G/[G,G] \simeq C_p\times C_p$.
There are precisely $3$ groups which belong to $\Phi_{10}$ if $p=3$.
For $p\ge5$,
the total number of non-isomorphic groups which belong to $\Phi_{10}$ is
\[
1+\gcd\{4,p-1\}+\gcd \{3,p-1\}.
\]
\end{theorem}

Note that, for $p=3$, the isoclinism family $\Phi_{10}$ consists of
the groups $\Phi_{10}(2111) a_r$ (where $r=0,1$) and $\Phi_{10}(1^5)$
\cite[page 621]{Ja}, which are just the groups $G(243,29)$,
$G(243,30)$ and $G(243,28)$ in  GAP code numbers respectively.

\bigskip
Now we turn to the other groups $G$ of order $243$ with
$\fn{Br}_{v,\bm{C}}(\bm{C}(G))=0$.  In this paper, we establish
the rationality of $k(G)$, where $k$ is any field with enough
roots of unity, except for those five groups which belong to
$\Phi_7$.

We state the main result of this paper as follows.

\begin{theorem} \label{thmain1}
Let $G$ be a group of order $243$ with exponent $e$. Let $k$ be a
field satisfying {\rm (i)} $\fn{char}k=3$, or {\rm (ii)}
$\fn{char}k\ne 3$ with $\zeta_e \in k$. If
$\fn{Br}_{v,\bm{C}}(\bm{C}(G))= 0$, then $k(G)$ is $k$-rational,
except possibly for the five groups $G$ which belong to the
isoclinism family $\Phi_7$, i.e. $G=G(243,i)$ with $56 \le i \le
60$.
\end{theorem}

The following two propositions provide some messages of $k(G)$
when $G$ is order $243$ and belongs to the isoclinism family
$\Phi_7$ or $\Phi_{10}$.

\begin{prop}[The case $\Phi_7$]\label{p1.15}
Let $G_1$ and $G_2$ be groups of order $243$ which belong to the
isoclinism family $\Phi_7$. If $k$ is a field with $\fn{char}k\neq
3$ and $\zeta_9\in k$, then $k(G_1)$ is $k$-isomorphic to
$k(G_2)$.

\end{prop}

\begin{prop}[The case $\Phi_{10}$]\label{p1.16}
Let $G_1$ and $G_2$ be groups of order $243$ which belong to the
isoclinism family $\Phi_{10}$. If $k$ is a field with
$\fn{char}k\neq 3$ and $\zeta_9\in k$, then $k(G_1)$ is
$k$-isomorphic to $k(G_2)$.

\end{prop}

We do not know whether $k(G)$ is $k$-rational if $G$ belongs to
$\Phi_7$. In our attempt to solve the case of groups in $\Phi_7$,
the situation is very similar to that of $\Phi_5$; the difference
of these two cases looks almost ``negligible". We did reach at
false proof for groups in $\Phi_7$ several times. But the
difficulty is not overcome anyhow.

It is possible to prove the rationality for many groups of order
$p^5$ (where $p \ge 5$) by the same method if $G$ doesn't belong
to the isoclinism family $\Phi_{5}$, $\Phi_{6}$, $\Phi_{7}$, or $\Phi_{10}$.


\bigskip
We explain briefly the idea of proving the above theorem. There
are $67$ groups of order $243$ in total. Except for the $3$ groups
which belong to $\Phi_{10}$, the $k$-rationality of $k(G)$ for
many groups $G$ may be obtained from the rationality criteria
given in Section 2. Indeed, $G$ belongs to $\Phi_1$ if and only if
$G$ is abelian, and hence $k(G)$ is $k$-rational in this case by
Theorem \ref{t2.2}. If $G$ belongs to $\Phi_2$, $\Phi_3$, $\Phi_4$
or $\Phi_9$, then there exists a normal abelian subgroup $N$ of
$G$ such that $G/N$ is cyclic of order $3$. Then $k(G)$ is
$k$-rational by Theorem \ref{t2.5}. If $G$ belongs to $\Phi_8$,
then there exists a normal cyclic subgroup $N$ of $G$ such that
$G/N$ is cyclic of order $9$. Hence $k(G)$ is $k$-rational by
Theorem \ref{tKa2}.

There remain $14$ groups $G$ in total which belong to $\Phi_5$,
$\Phi_6$ or $\Phi_7$, for which the $k$-rationality of $k(G)$
should be studied further. In studying the rationality problem of
these groups, new technical difficulties (different from the
situations in \cite{CK}, \cite{CHKP}, \cite{CHKK}) arise.
Fortunately we are able to discover new methods to solve these
difficulties (see Step 4 of Case 1 in Section 4, Steps 3 and 5 of
Case 1 in Section 5). We hope these methods will be useful for
other rationality problems.

This paper is organized as follows. We recall several rationality
criteria in Section 2. In Section 3, we use the database of groups
of order 243 in GAP and list generators and relations for $17$
groups which belong to $\Phi_5$, $\Phi_6$, $\Phi_7$ or
$\Phi_{10}$. We also exhibit a faithful representation of these
groups $G$ for which the rationality of $k(G)$ will be discussed
later. These faithful representations of $G$ are obtained as the
induced representations of some abelian normal subgroups of $G$ of
index $27$.
Section 4 and Section 5 consist of the proof of Theorem
\ref{thmain1} for $7$ groups $G(243,i)$, $3 \leq i\leq 9$, which
belong to $\Phi_6$ and for $2$ groups $G(243,65)$, $G(243,66)$
which belong to $\Phi_5$ respectively. The proof of Theorem
\ref{thmain1}, Proposition \ref{p1.15} and Proposition \ref{p1.16}
will be given in Section 6.

\begin{idef}{Standing Notations}
Throughout this paper, $k(x_1,\ldots,x_n)$ will be rational
function fields of $n$ variables over $k$.

We denote by $\zeta_n$ a primitive $n$-th root of unity. Whenever we
write $\zeta_n\in k$, it is understood that either $\fn{char}k=0$ or
$\fn{char}k>0$ with $\gcd \{n,\fn{char}k\}=1$. We always write
$\zeta$ for $\zeta_3$ for simplicity and $\eta$ for a primitive 9th
root of unity satisfying $\eta^3=\zeta$.

$I_n$ denotes the $n \times n$ identity matrix. If $G$ is a group,
$Z(G)$ denotes the center of $G$. If $g, h \in G$, define
$[g,h]=g^{-1}h^{-1}gh$. The exponent of a group $G$ is defined as
$\fn{lcm}\{\fn{ord}(g):g\in G\}$ where $\fn{ord}(g)$ is the order of
the element $g$. All the groups in this article are finite groups.
For emphasis, recall the definition $k(G)=k(x_g:g\in G)^G$ which was
defined in the first paragraph of this section. The group
$G(243,i)$, or $G(i)$ for short, is the $i$-th group of order 243 in
the GAP database. The version of GAP used in this paper is GAP4,
Version: 4.4.10 \cite{GAP}.
\end{idef}

\section{Preliminaries}

In this section, we record several results which will be used later.

\begin{theorem}[{\cite[Theorem 1]{HK}}] \label{t2.3}
Let $G$ be a finite group acting on $L(x_1,\ldots,x_n)$,
the rational function field of $n$ variables over a field $L$.
Suppose that\\
{\rm (i)} for any $\sigma \in G$, $\sigma(L)\subset L$;\\
{\rm (ii)} the restriction of the action of $G$ to $L$ is faithful; and\\
{\rm (iii)} for any $\sigma \in G$,
\[
\begin{pmatrix} \sigma (x_1) \\ \sigma (x_2) \\ \vdots \\ \sigma (x_n) \end{pmatrix}
= A(\sigma)\cdot \begin{pmatrix} x_1 \\ x_2 \\ \vdots \\ x_n \end{pmatrix} +B(\sigma)
\]
where $A(\sigma)\in GL_n(L)$ and $B(\sigma)$ is an $n\times 1$ matrix over $L$.\\
Then there exist $z_1,\ldots,z_n\in L(x_1,\ldots,x_n)$ such
that $L(x_1,\ldots,x_n)=L(z_1,\ldots,z_n)$
and $\sigma(z_i)=z_i$ for any $\sigma\in G$, any $1\le i\le n$.
\end{theorem}


\begin{theorem}[{\cite[Theorem $1'$]{HK}}]\label{thHK}
Let $G$ be a finite group acting on $L(x_1,\ldots,x_m)$,
the rational function field of $m$ variables over a field $L$.
Suppose that\\
{\rm (i)} for any $\sigma\in G$, $\sigma(L)\subset L$;\\
{\rm (ii)} the restriction of the actions of $G$ to $L$ is faithful; and\\
{\rm (iii)} for any $\sigma\in G$,
\begin{align*}
\left(
\begin{array}{c}
\sigma(x_1)\\ \sigma(x_2)\\ \vdots\\ \sigma(x_m)
\end{array}
\right)
=A(\sigma)
\left(
\begin{array}{c}
x_1\\ x_2\\ \vdots\\ x_m
\end{array}
\right)
\end{align*}
where $A(\sigma)\in GL_m(L)$.

Then $G$ acts on $L(x_1/x_m,\ldots,x_{m-1}/x_m)$ in the natural way.
Moreover, there exist $z_1,\ldots,z_m\in L(x_1,\ldots,x_m)$ such that
$L(x_1/x_m,\ldots,x_{m-1}/x_m)=L(z_1/z_m,\ldots,z_{m-1}/z_m)$
and $\sigma(z_i/z_m)=z_i/z_m$ for any $\sigma\in G$, and $1\leq i\leq m-1$.
In fact, $z_1,\ldots,z_m$ can be defined
\begin{align*}
z_j:=\sum_{i=1}^m \alpha_{ij}x_i,\quad {\rm for}\quad 1\leq j\leq m
\end{align*}
where $(\alpha_{ij})_{1\leq i,j\leq m}\in GL_m(L)$.
\end{theorem}
\begin{lemma}\label{lem4}
Let $G=\langle\sigma_1,\sigma_2\rangle\simeq C_3\times C_3$ act on
the rational function field $L(X,Y)$ with two variables $X,Y$ over $L$.
Suppose that \\
{\rm (i)} for any $\sigma\in G$, $\sigma(L)\subset L$;\\
{\rm (ii)} the restriction of the actions of $G$ to $L$ is faithful; and\\
{\rm (iii)} $G$ acts on $L(X,Y)$ by
\begin{align*}
\sigma_1 : X\mapsto Y\mapsto \frac{1}{XY},\quad
\sigma_2 : X\mapsto \frac{a}{\sigma_1(a)}X,
Y\mapsto \frac{\sigma_1(a)}{\sigma_1^2(a)} Y
\end{align*}
where $a\in L$ satisfies $a\cdot\sigma_2(a)\cdot\sigma_2^2(a)=1$.
Then there exist $Z,W\in L(X,Y)$ such that $L(X,Y)=L(Z,W)$ and
$\sigma(Z)=Z, \sigma(W)=W$ for any $\sigma\in G$.
\end{lemma}
\begin{proof}
We consider the action of $G$ on the rational function field
$L(x_1,x_2,x_3)$ with three variables $x_1,x_2,x_3$ over $L$ by
\begin{align*}
\sigma_1 : x_1\mapsto x_2\mapsto x_3\mapsto x_1,\quad
\sigma_2 : x_1\mapsto ax_1, x_2\mapsto \sigma_1(a)x_2,
x_3\mapsto \sigma_1^2(a)x_3.
\end{align*}
Then $\sigma_2^3(x_i)=x_i$ for $i=1,2,3$ because
$a\cdot\sigma_2(a)\cdot\sigma_2^2(a)=1$. Note that $\sigma_1
\sigma_2= \sigma_2 \sigma_1$.

Define $X=x_1/x_2$, $Y=x_2/x_3$. Then $G$ acts on $L(X,Y)$ as in
(iii). Apply Theorem \ref{thHK}. There exist $Z,W\in L(X,Y)$ such
that $L(X,Y)=L(Z,W)$ and $\sigma(Z)=Z, \sigma(W)=W$ for any
$\sigma\in G$ where $Z=z_1/z_2$, $W=z_2/z_3$ and $z_1,z_2,z_3$ are
as in Theorem \ref{thHK}.
\end{proof}

\begin{theorem}[{\cite[Theorem 3.1]{AHK}}] \label{t2.4}
Let $L$ be any field, $L(x)$ be the rational function field of one
variable over $L$, $G$ be a finite group acting on $L(x)$. Suppose
that, for any $\sigma \in G$, $\sigma(L)\subset L$ and
$\sigma(x)=a_\sigma x+b_\sigma$ where $a_\sigma, b_\sigma \in L$
and $a_\sigma\ne 0$. Then $L(x)^G=L^G(f)$ for some polynomial
$f\in L[x]$. In fact, if $m=\min \{\deg g(x):g(x)\in
L[x]^G\setminus L\}$, any polynomial $f\in L[x]^G$ with $\deg f=m$
satisfies the property $L(x)^G=L^G(f)$.
\end{theorem}

\begin{theorem}[{\cite[Theorem 1.4]{Ka1}}] \label{t2.5}
Let $k$ be a field, $G$ be a finite group.
Assume that\\
{\rm (i)} $G$ contains an abelian normal subgroup $H$ such that $G/H$
is cyclic of order $n$;\\
{\rm (ii)} $\bm{Z}[\zeta_n]$ is a unique factorization domain; and\\
{\rm (iii)} $k$ contains a primitive $e$-th root of unity
where $e$ is the exponent of $G$.\\
If $G\to GL(V)$ is any finite-dimensional representation of $G$
over $k$, then $k(V)^G$ is $k$-rational.
\end{theorem}

\begin{theorem}[{\cite[Theorem 1.8]{Ka2}}] \label{tKa2}
Let $n\geq 3$ and $G$ be a non-abelian $p$-group of order $p^n$
such that $G$ contains a cyclic subgroup of index $p^2$.
Assume that $k$ is any field satisfying that either
{\rm (i)} $\fn{char}k=p>0$, or
{\rm (ii)} $\fn{char}k\neq p$ and $k$ contains a primitive
$p^{n-2}$-th root of unity.
Then $k(G)$ is $k$-rational.
\end{theorem}

\begin{theorem} \label{t2.6}
Let $k$ be a field with $\fn{gcd} \{ \fn{char} k, n+1 \}=1$,
$A=(a_{ij})_{0\le i,j\le n} \in GL_{n+1}(k)$ and
$k(x_1,\ldots,x_n)$ be the rational function field of $n$
variables over $k$. Define $L_i=a_{i0}+\sum_{1\le j\le n}
a_{ij}x_j\in k[x_1,\ldots,x_n]$ for $0\le i\le n$ and define a
$k$-automorphism $\sigma:k(x_1,\ldots,x_n)\to k(x_1,\ldots,x_n)$
by $\sigma(x_i)=L_i/L_0$ for $1\le i\le n$. If the characteristic
polynomial of the matrix $A$ is $T^{n+1}-c\in k[T]$ where $c \in k
\setminus \{0 \}$, then there exist $y_1,\ldots,y_n \in
k(x_1,\ldots,x_n)$ such that $k(x_1,\ldots,x_n)=k(y_1,\ldots,y_n)$
and $\sigma(y_i)=y_{i+1}$ for $1\le i\le n-1$, $\sigma(y_n)=c/(y_1
y_2 \cdots y_n)$.
\end{theorem}

\begin{proof}
Consider another rational function field $k(u_0,u_1,\ldots,u_n)$.
Embed the field $k(x_1,\ldots,x_n)$ into $k(u_0,u_1,\ldots,u_n)$ by $x_i=u_i/u_0$ for $1\le i\le n$.

Define a $k$-automorphism $\Phi: k(u_0,\ldots,u_n)\to k(u_0,\ldots,u_n)$ by $\Phi(u_i)=\sum_{0\le j\le n} a_{ij} u_j$ for $0\le i\le n$.
It is clear that the restriction of $\Phi$ to $k(x_1,\ldots,x_n)$ is nothing but $\sigma$.

Since the characteristic polynomial of $A$ is the separable
polynomial $T^{n+1}-c$, the rational normal form of the matrix
$(a_{ij})_{0 \le i,j \le n}$ is the companion matrix of the
polynomial $T^{n+1}-c$. It follows that there exist
$v_0,v_1,\ldots,v_n$ such that (i) $v_i=\sum_{0\le j\le n}
b_{ij}u_j$ for $0\le i\le n$ where $(b_{ij})_{0\le i,j\le n} \in
GL_{n+1}(k)$, (ii) $\Phi(v_i)=v_{i+1}$ for $0\le i\le n-1$ and
$\Phi(v_n)=c v_0$.

Define $y_i=v_i/v_{i-1}$ for $1\le i\le n$. Then
$k(y_1,\ldots,y_n)=k(x_1,\ldots,x_n)$ and $\sigma(y_i)=y_{i+1}$
for $1\le i\le n-1$ and $\sigma(y_n)=c/(y_1y_2\cdots y_n)$.
\end{proof}

\begin{lemma}[{\cite[Lemma 3.6]{HKY}}] \label{l2.8}
Let $k$ be any field, $a\in k\backslash\{0\}$. Let $\sigma$ be a
$k$-automorphism acting on $k(x,y)$ by
\begin{align*}
\sigma &: x\mapsto y\mapsto \frac{a}{xy}.
\end{align*}
Then $k(x,y)^{\langle \sigma \rangle}$ is $k$-rational.
\end{lemma}

\section{Groups of order 243}

{}From the data base of GAP, there are $67$ groups of order $243$.
Their GAP codes are designated as $G(243,i)$ for $1 \le i \le 67$.
{}From now on, we abbreviate $G(243,i)$ as $G(i)$.\\

\noindent
\begin{tabular}{c|c|c|l|c}
family & rank & class & $G=G(i)=G(243,i)$, $i\in$
& \#\\\hline
$\Phi_1$ & $$ & $1$ & $\{ 1, 10, 23, 31, 48, 61, 67 \}$ ($G$: abelian)& $7$\\
$\Phi_2$ & $3$ & $2$ & $\{ 2, 11, 12, 21, 24, 32, 33, 34, 35, 36, 49, 50, 62, 63, 64 \}$ & $15$\\
$\Phi_3$ & $4$ & $3$ & $\{ 13, 14, 15, 16, 17, 18, 19, 20, 51, 52, 53, 54, 55 \}$ & $13$\\
$\Phi_4$ & $5$ & $2$ & $\{ 37, 38, 39, 40, 41, 42, 43, 44, 45, 46, 47 \}$ & $11$\\
$\Phi_5$ & $5$ & $2$ & $\{ 65, 66 \}$ & $2$\\
$\Phi_6$ & $5$ & $3$ & $\{ 3, 4, 5, 6, 7, 8, 9 \}$ & $7$\\
$\Phi_7$ & $5$ & $3$ & $\{ 56, 57, 58, 59, 60 \}$ & $5$\\
$\Phi_8$ & $5$ & $3$ & $\{ 22 \}$ ($G\simeq C_{27}\rtimes C_{9}$) & $1$\\
$\Phi_9$ & $5$ & $4$ & $\{ 25, 26, 27 \}$ ($G\simeq (C_9\times C_9)\rtimes C_3$)
& $3$\\
$\Phi_{10}$ & $5$ & $4$ & $\{ 28, 29, 30 \}$ & $3$\\\hline
total & $$ & $$ &  & $67$
\end{tabular}\\ \\

In the following we list the generators and relations of the $17$
groups $G=G(i)$ which belong to $\Phi_6$ ($3 \le i \le 9$),
$\Phi_5$ ($i=65, 66$), $\Phi_7$ ($56 \le i \le 60$) and
$\Phi_{10}$ ($28 \le i \le 30$). Then we give some faithful
representations of them over a field $k$ containing $\zeta_e$
where $e=\exp(G)$.

Note that $Z(G)\simeq C_3\times C_3$ (resp. $C_3$) when $G$
belongs to $\Phi_6$ (resp. $\Phi_5$, $\Phi_7$ or $\Phi_{10}$).

\bigskip
\begin{idef}{Case $\Phi_6$}
$G=G(i)=\langle f_1, f_2, f_3, f_4, f_5 \rangle$, $3\le i\le 9$ with
 $Z(G)=\langle f_4,f_5\rangle\simeq C_3\times C_3$, satisfying  common relations
 \[
 [f_2,f_1]=f_3, \quad [f_3,f_1]=f_4, \quad [f_3,f_2]=f_5, \quad f_3^3=f_4^3=f_5^3=1
 \] and extra relations
\begin{enumerate}
\item [(1)] ${\fn{for}} \ G(3)\ (\Phi_6(1^5))\quad\ :  f_1^3=1, ~ f_2^3=1$;
\item [(2)] ${\fn{for}} \ G(4)\ (\Phi_6(221)b_1):  f_2^3=1, ~ f_1^3=f_4$;
\item [(3)] ${\fn{for}} \ G(5)\ (\Phi_6(221)c_2):  f_1^3=f_2^3=f_4$;
\item [(4)] ${\fn{for}} \ G(6)\ (\Phi_6(221)d_1):  f_1^3=1, ~ f_2^{3}=f_4^2$;
\item [(5)] ${\fn{for}} \ G(7)\ (\Phi_6(221)a)\ :  f_1^3=f_4, f_2^{3}=f_4^2$;
\item [(6)] ${\fn{for}} \ G(8)\ (\Phi_6(221)c_1):  f_1^3=f_5, f_2^3=f_4$;
\item [(7)] ${\fn{for}} \ G(9)\ (\Phi_6(221)d_0):  f_1^{3}=f_5^2, f_2^3=f_4$.
\end{enumerate}
\end{idef}

\bigskip
\begin{idef}{Case $\Phi_5$}
$G=G(i)=\langle f_1, f_2, f_3, f_4, f_5 \rangle$, $i=65, 66$ with
 $Z(G)=\langle f_5\rangle\simeq C_3$ satisfying common relations
\begin{align*}
[f_2, f_1]&=[f_4,f_1]=[f_3,f_2]=f_5,\quad
[f_1,f_3]=[f_2, f_4]=[f_3, f_4]=1, \\
f_2^3&=f_3^3=f_4^3=f_5^3=1
\end{align*}
and extra relations
\begin{enumerate}
\item [(1)] ${\fn{for}} \ G(65)\ (\Phi_5(1^5))\ \ : f_1^3=1$;
\item [(2)] ${\fn{for}} \ G(66)\ (\Phi_5(2111):  f_1^3=f_5$.
\end{enumerate}
\end{idef}

\bigskip
\begin{idef}{Case $\Phi_7$}
$G=G(i)=\langle f_1, f_2, f_3, f_4, f_5 \rangle$, $56\le i\le 60$ with
$Z(G)=\langle f_5\rangle\simeq C_3$, satisfying common relations
\[
[f_2,f_1]=f_4,\quad [f_3,f_2]=[f_4,f_1]=f_5,\quad
[f_1,f_3]=[f_2,f_4]=[f_3, f_4]=1,\quad f_4^3=f_5^3=1
\]
and extra relations
\begin{enumerate}
\item [(1)] ${\fn{for}} \ G(56)\ (\Phi_7(2111)b_1): f_1^3=f_2^3=f_3^3=1$;
\item [(2)] ${\fn{for}} \ G(57)\ (\Phi_7(2111)b_2): f_1^3=f_3^3=1, ~ f_2^3=f_5$;
\item [(3)] ${\fn{for}} \ G(58)\ (\Phi_7(1^5))\qquad: f_1^3=f_3^3=1, ~ f_2^3=f_5^2$;
\item [(4)] ${\fn{for}} \ G(59)\ (\Phi_7(2111)a)\ : f_3^3=1, ~ f_1^3=f_5, f_2^3=f_5^2$;
\item [(5)] ${\fn{for}} \ G(60)\ (\Phi_7(2111)c)\ :f_1^3=f_2^3=1, ~  f_3^3=f_5$.
\end{enumerate}
\end{idef}

\bigskip
\begin{idef}{Case $\Phi_{10}$}
$G=G(i)=\langle f_1, f_2, f_3, f_4, f_5 \rangle$, $28\le i\le 30$ with
$Z(G)=\langle f_5\rangle\simeq C_3$, satisfying common relations
\begin{align*}
[f_2,f_1]&=f_3,\quad [f_3,f_1]=f_4,\quad [f_3,f_2]=[f_4,f_1]=f_5, \quad
[f_2,f_4]=[f_3, f_4]=1,\\
f_4^3&=f_5^3=1,\quad f_2^3=f_4^2,\quad f_3^3=f_5^2
\end{align*}
and extra relations
\begin{enumerate}
\item [(1)] ${\fn{for}} \ G(28)\ (\Phi_{10}(1^5))\qquad : f_1^3=1$;
\item [(2)] ${\fn{for}} \ G(29)\ (\Phi_{10}(2111)a_0): f_1^3=f_5$;
\item [(3)] ${\fn{for}} \ G(30)\ (\Phi_{10}(2111)a_1): f_1^3=f_5^2$.
\end{enumerate}
\end{idef}

\bigskip
We now give some faithful representations
for groups which belong to $\Phi_6$, $\Phi_5$, $\Phi_7$ and $\Phi_{10}$
respectively.
Let $I_n$ be the $n\times n$ identity matrix,
\begin{align*}
c_3^{(i)}&=\left[
\begin{array}{ccc}
0 & 0 & 1 \\
\zeta^i & 0 & 0 \\
0 & 1 & 0
\end{array}
\right],\quad
c_3=c_3^{(0)}=\left[
\begin{array}{ccc}
0 & 0 & 1 \\
1 & 0 & 0 \\
0 & 1 & 0
\end{array}
\right],\\
d_3&=\left[
\begin{array}{ccc}
1 & 0 & 0 \\
0 & \zeta & 0 \\
0 & 0 & \zeta^2
\end{array}
\right],\quad
d_9^{(i)}=\left[
\begin{array}{ccc}
1 & 0 & 0 \\
0 & \eta^i & 0 \\
0 & 0 & \eta^{2i}
\end{array}
\right],\quad
e_3^{(i)}=\left[
\begin{array}{ccc}
1 & 0 & 0 \\
0 & 1 & 0 \\
0 & 0 & \zeta^i
\end{array}
\right]
\end{align*}
where $\eta^3=\zeta$, $\zeta^3=1$.


\begin{idef}{Case $\Phi_6$}
For groups $G=G(i)$, $(3\leq i\leq 9)$ which belong to $\Phi_6$,
we take the following $6$-dimensional faithful representations
where
\[
f_3  \mapsto \begin{bmatrix} d_3 & {\bf 0} \\ {\bf 0} & d_3 \end{bmatrix}, \\
f_4 \mapsto \begin{bmatrix} \zeta I_3 & {\bf 0} \\ {\bf 0} & I_3
\end{bmatrix},  f_5 \mapsto \begin{bmatrix} I_3 & {\bf 0} \\ {\bf 0} &
\zeta I_3 \end{bmatrix}
\]
are common for each $3\leq i\leq 9$ and
\end{idef}
\begin{enumerate}
\item [(1)] ${\fn{for}} \ G(3):\
f_1\mapsto\left[
\begin{array}{cc}
c_3 & \mathbf{0}\\
\mathbf{0} & e_3^{(2)}
\end{array}
\right],
f_2\mapsto\left[
\begin{array}{cc}
e_3^{(1)} & \mathbf{0} \\
\mathbf{0} & c_3
\end{array}
\right];$
\item [(2)] ${\fn{for}} \ G(4):\
f_1\mapsto\left[
\begin{array}{cc}
c_3^{(1)} & \mathbf{0}\\
\mathbf{0} & e_3^{(2)}
\end{array}
\right],
f_2\mapsto\left[
\begin{array}{cc}
e_3^{(1)} & \mathbf{0} \\
\mathbf{0} & c_3
\end{array}
\right];$
\item [(3)] ${\fn{for}} \ G(5):\
f_1\mapsto\left[
\begin{array}{cc}
c_3^{(1)} & \mathbf{0}\\
\mathbf{0} & e_3^{(2)}
\end{array}
\right],
f_2\mapsto\left[
\begin{array}{cc}
\eta e_3^{(1)} & \mathbf{0} \\
\mathbf{0} & c_3
\end{array}
\right];$
\item [(4)] ${\fn{for}} \ G(6):\
f_1\mapsto\left[
\begin{array}{cc}
c_3 & \mathbf{0}\\
\mathbf{0} & e_3^{(2)}
\end{array}
\right],
f_2\mapsto\left[
\begin{array}{cc}
\eta^2 e_3^{(1)} & \mathbf{0} \\
\mathbf{0} & c_3
\end{array}
\right];$
\item [(5)] ${\fn{for}} \ G(7):\
f_1\mapsto\left[
\begin{array}{cc}
c_3^{(1)} & \mathbf{0}\\
\mathbf{0} & e_3^{(2)}
\end{array}
\right],
f_2\mapsto\left[
\begin{array}{cc}
\eta^2 e_3^{(1)} & \mathbf{0} \\
\mathbf{0} & c_3
\end{array}
\right];$
\item [(6)] ${\fn{for}} \ G(8):\
f_1\mapsto\left[
\begin{array}{cc}
c_3 & \mathbf{0}\\
\mathbf{0} & \eta e_3^{(2)}
\end{array}
\right],
f_2\mapsto\left[
\begin{array}{cc}
\eta e_3^{(1)} & \mathbf{0} \\
\mathbf{0} & c_3
\end{array}
\right];$
\item [(7)] ${\fn{for}} \ G(9):\
f_1\mapsto\left[
\begin{array}{cc}
c_3 & \mathbf{0}\\
\mathbf{0} & \eta^2 e_3^{(2)}
\end{array}
\right],
f_2\mapsto\left[
\begin{array}{cc}
\eta e_3^{(1)} & \mathbf{0} \\
\mathbf{0} & c_3
\end{array}
\right].$
\end{enumerate}

\begin{idef}{Case $\Phi_5$}
For groups $G=G(i)$, $(i=65, 66)$ which belong to $\Phi_5$,
we take the following $9$-dimensional faithful representations
which are induced from
a linear character on $\langle f_3,f_4,f_5\rangle$ where
\begin{align*}
f_2\mapsto\left[
\begin{array}{ccc}
c_3 & \mathbf{0} & \mathbf{0} \\
\mathbf{0} & \zeta c_3 & \mathbf{0} \\
\mathbf{0} & \mathbf{0} & \zeta^2 c_3
\end{array}
\right],
f_3\mapsto\left[
\begin{array}{ccc}
d_3 & \mathbf{0} & \mathbf{0} \\
\mathbf{0} & d_3 & \mathbf{0} \\
\mathbf{0} & \mathbf{0} & d_3
\end{array}
\right],
f_4\mapsto\left[
\begin{array}{ccc}
I_3 & \mathbf{0} & \mathbf{0} \\
\mathbf{0} & \zeta I_3 & \mathbf{0} \\
\mathbf{0} & \mathbf{0} & \zeta^2 I_3
\end{array}
\right],
f_5\mapsto\zeta I_9
\end{align*}
are common for each $i=65,66$ and
\end{idef}
\begin{enumerate}
\item [(1)] ${\fn{for}} \ G(65):\
f_1\mapsto\left[
\begin{array}{ccc}
\mathbf{0} & \mathbf{0} & I_3 \\
I_3 & \mathbf{0} & \mathbf{0} \\
\mathbf{0} & I_3 & \mathbf{0}
\end{array}
\right];$
\item [(2)] ${\fn{for}} \ G(66):\
f_1\mapsto\left[
\begin{array}{ccc}
\mathbf{0} & \mathbf{0} & I_3 \\
\zeta I_3 & \mathbf{0} & \mathbf{0} \\
\mathbf{0} & I_3 & \mathbf{0}
\end{array}
\right]$.
\end{enumerate}

\begin{idef}{Case $\Phi_7$}
For groups $G=G(i)$, $(56\leq i\leq 60)$ which belong to $\Phi_7$,
we take the following $9$-dimensional faithful representations
which are induced from
a linear character on $\langle f_3,f_4,f_5\rangle$ where
\begin{align*}
f_4\mapsto\left[
\begin{array}{ccc}
I_3 & \mathbf{0} & \mathbf{0} \\
\mathbf{0} & \zeta I_3 & \mathbf{0} \\
\mathbf{0} & \mathbf{0} & \zeta^2 I_3
\end{array}
\right],
f_5\mapsto\zeta I_9
\end{align*}
are common for each $56\leq i\leq 60$ and
\end{idef}
\begin{enumerate}
\item [(1)] ${\fn{for}} \ G(56):\
f_1\mapsto\left[
\begin{array}{ccc}
\mathbf{0} & \mathbf{0} & I_3 \\
I_3 & \mathbf{0} & \mathbf{0} \\
\mathbf{0} & I_3 & \mathbf{0}
\end{array}
\right],
f_2\mapsto\left[
\begin{array}{ccc}
c_3 & \mathbf{0} & \mathbf{0} \\
\mathbf{0} & c_3 & \mathbf{0} \\
\mathbf{0} & \mathbf{0} & \zeta c_3
\end{array}
\right],
f_3\mapsto\left[
\begin{array}{ccc}
d_3 & \mathbf{0} & \mathbf{0} \\
\mathbf{0} & d_3 & \mathbf{0} \\
\mathbf{0} & \mathbf{0} & d_3
\end{array}
\right];$
\item [(2)] ${\fn{for}} \ G(57):\
f_1\mapsto\left[
\begin{array}{ccc}
\mathbf{0} & \mathbf{0} & I_3 \\
I_3 & \mathbf{0} & \mathbf{0} \\
\mathbf{0} & I_3 & \mathbf{0}
\end{array}
\right],
f_2\mapsto\left[
\begin{array}{ccc}
c_3^{(1)} & \mathbf{0} & \mathbf{0} \\
\mathbf{0} & c_3^{(1)} & \mathbf{0} \\
\mathbf{0} & \mathbf{0} & \zeta c_3^{(1)}
\end{array}
\right],
f_3\mapsto\left[
\begin{array}{ccc}
d_3 & \mathbf{0} & \mathbf{0} \\
\mathbf{0} & d_3 & \mathbf{0} \\
\mathbf{0} & \mathbf{0} & d_3
\end{array}
\right];$
\item [(3)] ${\fn{for}} \ G(58):\
f_1\mapsto\left[
\begin{array}{ccc}
\mathbf{0} & \mathbf{0} & I_3 \\
I_3 & \mathbf{0} & \mathbf{0} \\
\mathbf{0} & I_3 & \mathbf{0}
\end{array}
\right],
f_2\mapsto\left[
\begin{array}{ccc}
c_3^{(2)} & \mathbf{0} & \mathbf{0} \\
\mathbf{0} & c_3^{(2)} & \mathbf{0} \\
\mathbf{0} & \mathbf{0} & \zeta c_3^{(2)}
\end{array}
\right],
f_3\mapsto\left[
\begin{array}{ccc}
d_3 & \mathbf{0} & \mathbf{0} \\
\mathbf{0} & d_3 & \mathbf{0} \\
\mathbf{0} & \mathbf{0} & d_3
\end{array}
\right];$
\item [(4)] ${\fn{for}} \ G(59):\
f_1\mapsto\left[
\begin{array}{ccc}
\mathbf{0} & \mathbf{0} & I_3 \\
\zeta I_3 & \mathbf{0} & \mathbf{0} \\
\mathbf{0} & I_3 & \mathbf{0}
\end{array}
\right],
f_2\mapsto\left[
\begin{array}{ccc}
c_3^{(2)} & \mathbf{0} & \mathbf{0} \\
\mathbf{0} & c_3^{(2)} & \mathbf{0} \\
\mathbf{0} & \mathbf{0} & \zeta c_3^{(2)}
\end{array}
\right],
f_3\mapsto\left[
\begin{array}{ccc}
d_3 & \mathbf{0} & \mathbf{0} \\
\mathbf{0} & d_3 & \mathbf{0} \\
\mathbf{0} & \mathbf{0} & d_3
\end{array}
\right];$
\item [(5)] ${\fn{for}} \ G(60):\
f_1\mapsto\left[
\begin{array}{ccc}
\mathbf{0} & \mathbf{0} & I_3 \\
I_3 & \mathbf{0} & \mathbf{0} \\
\mathbf{0} & I_3 & \mathbf{0}
\end{array}
\right],
f_2\mapsto\left[
\begin{array}{ccc}
c_3 & \mathbf{0} & \mathbf{0} \\
\mathbf{0} & c_3 & \mathbf{0} \\
\mathbf{0} & \mathbf{0} & \zeta c_3
\end{array}
\right],
f_3\mapsto\left[
\begin{array}{ccc}
\eta d_3 & \mathbf{0} & \mathbf{0} \\
\mathbf{0} & \eta d_3 & \mathbf{0} \\
\mathbf{0} & \mathbf{0} & \eta d_3
\end{array}
\right].$
\end{enumerate}

\begin{idef}{Case $\Phi_{10}$}
For groups $G=G(i)$, $(28\leq i\leq 30)$ which belong to $\Phi_{10}$,
we take the following $9$-dimensional faithful representations
which are induced from
a linear character on $\langle f_3,f_4,f_5\rangle$ where
\begin{align*}
&f_2\mapsto\left[
\begin{array}{ccc}
\mathbf{0} & d_9^{(5)} & \mathbf{0} \\
\mathbf{0} & \mathbf{0} & d_9^{(2)} \\
d_9^{(8)} & \mathbf{0} & \mathbf{0}
\end{array}
\right],
f_3\mapsto\left[
\begin{array}{ccc}
\eta^8 e^{(1)} & \mathbf{0} & \mathbf{0} \\
\mathbf{0} & \eta^5 e^{(1)} & \mathbf{0} \\
\mathbf{0} & \mathbf{0} & \eta^2 e^{(1)}
\end{array}
\right],
f_4\mapsto\left[
\begin{array}{ccc}
d_3 & \mathbf{0} & \mathbf{0} \\
\mathbf{0} & d_3 & \mathbf{0} \\
\mathbf{0} & \mathbf{0} & d_3
\end{array}
\right],\\
&f_5\mapsto\zeta I_9
\end{align*}
are common for each $28\leq i\leq 30$ and
\end{idef}
\begin{enumerate}
\item [(1)] ${\fn{for}} \ G(28):\
f_1\mapsto\left[
\begin{array}{ccc}
c_3 & \mathbf{0} & \mathbf{0} \\
\mathbf{0} & c_3 & \mathbf{0} \\
\mathbf{0} & \mathbf{0} & c_3
\end{array}
\right];$
\item [(2)] ${\fn{for}} \ G(29):\
f_1\mapsto\left[
\begin{array}{ccc}
c_3^{(1)} & \mathbf{0} & \mathbf{0} \\
\mathbf{0} & c_3^{(1)} & \mathbf{0} \\
\mathbf{0} & \mathbf{0} & c_3^{(1)}
\end{array}
\right];$
\item [(3)] ${\fn{for}} \ G(30):\
f_1\mapsto\left[
\begin{array}{ccc}
c_3^{(2)} & \mathbf{0} & \mathbf{0} \\
\mathbf{0} & c_3^{(2)} & \mathbf{0} \\
\mathbf{0} & \mathbf{0} & c_3^{(2)}
\end{array}
\right].$
\end{enumerate}

\section{The Case $\Phi_6$: \boldmath{$G(i)$, $3\le i\le 9$}}

Let $G=G(i)$ be the $i$-th group of order 243 in
the GAP database where $3\le i\le 9$.
They belong to the isoclinism family $\Phi_6$.
In this section, we will prove that $k(G(i))$ is $k$-rational
for $3\leq i\leq 9$.

Recall that $\zeta=\zeta_3$ is a primitive $3$-rd
root of unity belonging to $k$,
and $\eta$ is a primitive $9$-th
root of unity satisfying $\eta^3=\zeta$.

\bigskip
\begin{idef}{Case 1} $G=G(3)$.

\medskip
Step 1.

We will find a faithful representation $G \rightarrow GL(V_3)$
according to the matrices as in Section 3. More precisely,
if $\{x_{11},x_{12},x_{13},x_{21},x_{22},x_{23} \}$ is a dual basis of $V_3$, we
choose the faithful representation $G \rightarrow GL(V_3)$ such that
$G$ acts on $\oplus_{1 \le i \le 2\atop 1\leq j\leq 3}k \cdot x_{ij}$
by the matrices as in Section 3. It follows that
$k(V_3)=k(x_{11},x_{12},x_{13},x_{21},x_{22},x_{23})$ and $G$ acts on it as follows.
%
\begin{align*}
f_1 : \ &
x_{11}\mapsto x_{12},x_{12}\mapsto x_{13},x_{13}\mapsto x_{11},
x_{21}\mapsto x_{21},x_{22}\mapsto x_{22},x_{23}\mapsto \zeta^2 x_{23},\\
f_2 : \ &
x_{11}\mapsto x_{11},x_{12}\mapsto x_{12},x_{13}\mapsto \zeta x_{13},
x_{21}\mapsto x_{22},x_{22}\mapsto x_{23},x_{23}\mapsto x_{21},\\
f_3 : \ &
x_{11}\mapsto x_{11},x_{12}\mapsto \zeta x_{12},x_{13}\mapsto \zeta^2 x_{13},
x_{21}\mapsto x_{21},x_{22}\mapsto \zeta x_{22},x_{23}\mapsto \zeta^2 x_{23},\\
f_4 : \ &
x_{11}\mapsto \zeta x_{11},x_{12}\mapsto \zeta x_{12},x_{13}\mapsto \zeta x_{13},
x_{21}\mapsto x_{21},x_{22}\mapsto x_{22},x_{23}\mapsto x_{23},\\
f_5 : \ &
x_{11}\mapsto x_{11},x_{12}\mapsto x_{12},x_{13}\mapsto x_{13},
x_{21}\mapsto \zeta x_{21},x_{22}\mapsto \zeta x_{22},x_{23}\mapsto \zeta x_{23}.
\end{align*}
Since $V_3$ is chosen such that it is a direct sum of inequivalent
irreducible representations, we may apply Theorem \ref{t2.3}.
We find that $k(G)$ is rational over $k(V_3)^G$. Once we show that
$k(V_3)^G$ is $k$-rational, it follows that $k(G)$ is also $k$-rational.

The remaining steps of this case are devoted to proving
$k(x_{11},x_{12},x_{13},x_{21},x_{22},x_{23})^G$ is $k$-rational.

\medskip
Step 2.

Define $y_{11}=\tfrac{x_{11}}{x_{12}}$, $y_{12}=\tfrac{x_{12}}{x_{13}}$,
$y_{13}=x_{13}$, $y_{21}=\tfrac{x_{21}}{x_{22}}$, $y_{22}=\tfrac{x_{22}}{x_{23}}$,
$y_{23}=x_{23}$.
Then $k(x_{11},x_{12},x_{13},x_{21},x_{22},x_{23})$ $=$
$k(y_{11},y_{12},y_{13},y_{21},y_{22},y_{23})$ and
\begin{align*}
f_1 : \ &
y_{11} \mapsto y_{12}, y_{12} \mapsto \tfrac{1}{y_{11}y_{12}},
y_{13} \mapsto y_{11}y_{12}y_{13},
y_{21} \mapsto y_{21}, y_{22} \mapsto \zeta y_{22},
y_{23} \mapsto \zeta^2 y_{23},\\
f_2 : \ & y_{11} \mapsto y_{11}, y_{12} \mapsto \zeta^2 y_{12},
y_{13} \mapsto \zeta y_{13}, y_{21} \mapsto y_{22}, y_{22} \mapsto
\tfrac{1}{y_{21} y_{22}},
y_{23} \mapsto y_{21} y_{22} y_{23},\\
f_3 : \ &
y_{11} \mapsto \zeta^2 y_{11}, y_{12} \mapsto \zeta^2 y_{12}, y_{13} \mapsto \zeta^2 y_{13},
y_{21} \mapsto \zeta^2 y_{21},  y_{22} \mapsto \zeta^2 y_{22}, y_{23} \mapsto \zeta^2 y_{23},\\
f_4 : \ &
y_{11} \mapsto y_{11}, y_{12} \mapsto y_{12}, y_{13} \mapsto \zeta y_{13},
y_{21} \mapsto y_{21}, y_{22} \mapsto y_{22},  y_{23} \mapsto y_{23},\\
f_5 : \ &
y_{11} \mapsto y_{11}, y_{12} \mapsto y_{12}, y_{13} \mapsto y_{13},
y_{21} \mapsto y_{21}, y_{22} \mapsto y_{22},  y_{23} \mapsto \zeta y_{23}.
\end{align*}
Apply Theorem 2.5 twice to
$k(y_{11},y_{12},y_{13},y_{21},y_{22},y_{23})
=k(y_{11},y_{12},y_{21},y_{22})(y_{13},y_{23})$,
it suffices to show that $k(y_{11},y_{12},y_{21},y_{22})^G$ is $k$-rational.

\medskip
Step 3.

Since $f_4$ and $f_5$ act trivially on $y_{11},y_{12},y_{21}$ and
$y_{22}$, we find that $k(y_{11},y_{12},y_{21},y_{22})^G$ $=$
$k(y_{11},y_{12},y_{21},y_{22})^{\langle f_1,f_2,f_3\rangle}$.
Define
\begin{align*}
&z_{1}=\frac{1}{y_{11} y_{21} y_{22}},
z_{2}=\frac{y_{12}}{y_{11}},
z_{3}=\frac{y_{21}}{y_{11}},
z_{4}=\frac{y_{22}}{y_{11}}.
\end{align*}
Because these $z_i$ are fixed by $f_3$ and the determinant of the
exponents of $z_i$ with respect to $y_j$ is $-3$, it is easy to
see that
\begin{align*}
k(y_{11},y_{12},y_{21},y_{22})^{\langle f_3\rangle}
=k(z_{1},z_{2},z_{3},z_{4}).
\end{align*}
Note that
\begin{align*}
f_1 : \ &z_1 \mapsto \frac{\zeta^2 z_1}{z_2},
z_2 \mapsto \frac{z_1z_3z_4}{z_2^2},
z_3 \mapsto \frac{z_3}{z_2},
z_4 \mapsto \frac{\zeta z_4}{z_2},\\
f_2 : \ &z_1 \mapsto z_3, z_2 \mapsto \zeta^2 z_2,
z_3 \mapsto z_4, z_4 \mapsto z_1.
\end{align*}
Define
\begin{align*}
w_1=\frac{z_1+z_3+z_4}{3},
w_2=z_2,
w_3=\frac{z_1+\zeta^2 z_3+\zeta z_4}{3},
w_4=\frac{z_1+\zeta z_3+\zeta^2 z_4}{3}.
\end{align*}
Then $k(z_1,z_2,z_3,z_4)
=k(w_1,w_2,w_3,w_4)$ and
\begin{align*}
f_1 : \ &w_1 \mapsto \frac{\zeta^2 w_4}{w_2},
w_2 \mapsto \frac{(w_1+w_3+w_4)
(w_1+\zeta w_3+\zeta^2 w_4)(w_1+\zeta^2 w_3+\zeta w_4)}{w_2^2},\\
&w_3 \mapsto \frac{\zeta^2 w_1}{w_2},
w_4 \mapsto \frac{\zeta^2 w_3}{w_2},\\
f_2 : \ &w_1 \mapsto w_1, w_2 \mapsto \zeta^2 w_2,
w_3 \mapsto \zeta w_3, w_4 \mapsto \zeta^2 w_4.
\end{align*}
Define
\begin{align*}
p_1=w_1,
p_2=\frac{\zeta^2 w_4}{w_2},
p_3=\frac{w_3^2}{w_1w_4},
p_4=\frac{w_1^2}{w_3w_4}.
\end{align*}
By the determinant trick again, we find that
\begin{align*}
k(w_1,w_2,w_3,w_4)^{\langle f_2\rangle}
=k(p_1,p_2,p_3,p_4)
\end{align*}
and
\begin{align*}
f_1 : \ p_1 \mapsto p_2,
p_2 \mapsto \frac{p_3p_4}
{p_1p_2(1-3p_3p_4+p_3^2p_4+p_3p_4^2)},
p_3 \mapsto p_4,
p_4 \mapsto \frac{1}{p_3p_4}.
\end{align*}
It remains to show that $k(p_1,p_2,p_3,p_4)^{\langle f_1\rangle}$
is $k$-rational.

\medskip
Step 4.

Define
\begin{align*}
q_1=p_1, q_2=p_2, q_3=\frac{1}{1+p_3+p_3p_4},
q_4=\frac{p_3}{1+p_3+p_3p_4}.
\end{align*}
Then $k(p_1,p_2,p_3,p_4)=k(q_1,q_2,q_3,q_4)$ and
\begin{align*}
f_1 : \ &q_1 \mapsto q_2,
q_2 \mapsto \frac{q_3q_4(1-q_3-q_4)}
{q_1q_2(q_3-2q_3^2+q_3^3-5q_3q_4+6q_3^2q_4+q_4^2+3q_3q_4^2-q_4^3)},\\
&q_3 \mapsto q_4, q_4 \mapsto 1-q_3-q_4.
\end{align*}
Define
\begin{align*}
r_1=q_1, r_2=q_2, r_3=q_3+\zeta^2 q_4+\zeta (1-q_3-q_4),
r_4=q_3+\zeta q_4+\zeta^2 (1-q_3-q_4).
\end{align*}
Then $k(q_1,q_2,q_3,q_4)=k(r_1,r_2,r_3,r_4)$ and
\begin{align*}
f_1 : \ r_1 \mapsto r_2,
r_2 \mapsto \frac{1+r_3^3-3r_3r_4+r_4^3}{3r_1r_2(3r_3r_4-r_4^3(\zeta+2)+r_3^3(\zeta-1))},
r_3 \mapsto \zeta r_3, r_4 \mapsto \zeta^2 r_4.
\end{align*}
We also define
\begin{align*}
s_1=\frac{r_3}{1+r_3+r_4}r_1,
s_2=f_1(s_1)=\frac{\zeta r_3}{1+\zeta r_3+\zeta^2 r_4}r_2,
s_3=r_3, s_4=r_4.
\end{align*}
Then $k(r_1,r_2,r_3,r_4)=k(s_1,s_2,s_3,s_4)$ and
\begin{align*}
f_1 :  s_1 \mapsto s_2, s_2 \mapsto
\frac{s_3^3}{3s_1s_2(3s_3s_4-s_4^3(\zeta+2)+s_3^3(\zeta-1))}, s_3
\mapsto \zeta s_3, s_4 \mapsto \zeta^2 s_4.
\end{align*}
Define $t_1=s_1, t_2=s_2$,
\begin{align*}
t_3=f_1(s_2)
=\frac{(\frac{s_3}{s_4})^3}{3s_1s_2(3(\frac{s_3}{s_4})(\frac{1}{s_4})-(\zeta+2)+(\frac{s_3}{s_4})^3(\zeta-1))},
\end{align*}
$t_4=\frac{s_3}{s_4}$.
Then $k(s_1,s_2,s_3,s_4)=k(t_1,t_2,t_3,t_4)$ and
\begin{align*}
f_1 : \ t_1 \mapsto t_2,
t_2 \mapsto t_3,
t_3 \mapsto t_1, t_4\mapsto \zeta^2 t_4.
\end{align*}
By Theorem \ref{t2.3}, $k(t_1,t_2,t_3,t_4)^{\langle f_1\rangle}$
is rational over $k(t_4)^{\langle f_1\rangle}$. Since
$k(t_4)^{\langle f_1\rangle}=k(t_4^3)$ is $k$-rational, it follows
that $k(t_1,t_2,t_3,t_4)^{\langle f_1\rangle}$ is $k$-rational.
Hence we conclude that $k(G(3))$ is $k$-rational.\\
\end{idef}

\bigskip
\begin{idef}{Case 2} $G=G(4), G(5), G(6), G(7), G(8), G(9)$.

\medskip
For $G=G(i)$, $(i=4,5,6,7,8,9)$, apply the same method as in Case
1: $G=G(3)$. We finally reduce the question to the rationality of
$k(t_1,t_2,t_3,t_4)^{\langle f_1\rangle}$ where the action of
$\langle f_1\rangle$ on $k(t_1,t_2,t_3,t_4)$ is given by
\begin{align*}
f_1 : \ t_1 \mapsto \zeta^{-j} t_2,
t_2 \mapsto t_3,
t_3 \mapsto \zeta^j t_1, t_4\mapsto \zeta^2 t_4
\end{align*}
where
\begin{align*}
j=\begin{cases}0\quad\ {\rm if}\quad i=3,6,8,9,\\
1 \quad\ {\rm if}\quad i=4,5,7.\end{cases}
\end{align*}
Thus $k(G(i))$ is $k$-rational for $i=6,8,9$ by the same reason as
in Case 1. When $i=4,5,7$, define
\begin{align*}
u_1=\zeta t_1, u_2=t_2, u_3=t_3, u_4=t_4.
\end{align*}
Then $k(u_1,u_2,u_3,u_4)=k(t_1,t_2,t_3,t_4)$ and
\begin{align*}
f_1 : \ u_1 \mapsto u_2,
u_2 \mapsto u_3,
u_3 \mapsto u_1, u_4\mapsto \zeta^2 u_4.
\end{align*}
Hence $k(G(i))$ is also $k$-rational for $i=4,5,7$.
\end{idef}


\section{The Case $\Phi_5$: \boldmath{$G(i)$, $65\le i\le 66$}}

Let $G=G(i)$ be the $i$-th group of order 243 in
the GAP database where $i=65, 66$.
They belong to the isoclinism family $\Phi_5$.
In this section, we will prove that $k(G(i))$ is $k$-rational
for $i=65, 66$.
We will prove the rationality of $G(65)$ first,
and deduce the rationality of $G(66)$ from it.
Note that $G(65)$ and $G(66)$ are
extraspecial $3$-groups of order $243$.

\bigskip
\begin{idef}{Case 1} $G=G(65)$.

\medskip
Step 1.

Choose a faithful representation $G=G(65) \rightarrow GL(V_{65})$
according to the matrices as in Section 3. By Theorem \ref{t2.3},
it remains to show that $k(V_{65})^G$ is $k$-rational. The action
of $G$ on
$k(V_{65})=k(x_{11},x_{12},x_{13},x_{21},x_{22},x_{23},x_{31},x_{32},x_{33})$
is given as follows.
\begin{align*}
f_1 : \ &
x_{11}\mapsto x_{21},x_{12}\mapsto x_{22},x_{13}\mapsto x_{23},
x_{21}\mapsto x_{31},x_{22}\mapsto x_{32},x_{23}\mapsto x_{33},\\
&x_{31}\mapsto x_{11},x_{32}\mapsto x_{12},x_{33}\mapsto x_{13},\\
f_2 : \ &
x_{11}\mapsto x_{12},x_{12}\mapsto x_{13},x_{13}\mapsto x_{11},
x_{21}\mapsto \zeta x_{22},x_{22}\mapsto \zeta x_{23},x_{23}\mapsto \zeta x_{21},\\
&x_{31}\mapsto \zeta^2 x_{32},x_{32}\mapsto \zeta^2 x_{33},x_{33}\mapsto \zeta^2 x_{31},\\
f_3 : \ &
x_{11}\mapsto x_{11},x_{12}\mapsto \zeta x_{12},x_{13}\mapsto \zeta^2 x_{13},
x_{21}\mapsto x_{21},x_{22}\mapsto \zeta x_{22},x_{23}\mapsto \zeta^2 x_{23},\\
&x_{31}\mapsto x_{31},x_{32}\mapsto \zeta x_{32},x_{33}\mapsto \zeta^2 x_{33},\\
f_4 : \ &
x_{11}\mapsto x_{11},x_{12}\mapsto x_{12},x_{13}\mapsto x_{13},
x_{21}\mapsto \zeta x_{21},x_{22}\mapsto \zeta x_{22},x_{23}\mapsto \zeta x_{23},\\
&x_{31}\mapsto \zeta^2 x_{31},x_{32}\mapsto \zeta^2 x_{32},x_{33}\mapsto \zeta^2 x_{33},\\
f_5 : \ &
x_{11}\mapsto \zeta x_{11},x_{12}\mapsto \zeta x_{12},x_{13}\mapsto \zeta x_{13},
x_{21}\mapsto \zeta x_{21},x_{22}\mapsto \zeta x_{22},x_{23}\mapsto \zeta x_{23},\\
&x_{31}\mapsto \zeta x_{31},x_{32}\mapsto \zeta x_{32},x_{33}\mapsto \zeta x_{33}.
\end{align*}

\medskip
Step 2.

Define $y_{11}=\frac{x_{11}}{x_{12}}$, $y_{12}=\frac{x_{12}}{x_{13}}$,
$y_{13}=x_{13}$, $y_{21}=\frac{x_{21}}{x_{22}}$, $y_{22}=\frac{x_{22}}{x_{23}}$,
$y_{23}=x_{23}$, $y_{31}=\frac{x_{31}}{x_{32}}$, $y_{32}=\frac{x_{32}}{x_{33}}$,
$y_{33}=x_{33}$.
Then
\[
k(x_{11},x_{12},x_{13},x_{21},x_{22},x_{23},x_{31},x_{32},x_{33})
=k(y_{11},y_{12},y_{13},y_{21},y_{22},y_{23},y_{31},y_{32},y_{33})
\]
and
\begin{align*}
f_1 : \ &
y_{11} \mapsto y_{21}, y_{12} \mapsto y_{22}, y_{13} \mapsto y_{23},
y_{21} \mapsto y_{31}, y_{22} \mapsto y_{32},  y_{23} \mapsto y_{33},\\
&y_{31} \mapsto y_{11}, y_{32} \mapsto y_{12}, y_{33} \mapsto y_{13},\\
f_2 : \ &
y_{11} \mapsto y_{12}, y_{12} \mapsto \tfrac{1}{y_{11} y_{12}},
y_{13} \mapsto y_{11} y_{12} y_{13}, y_{21} \mapsto y_{22},
y_{22} \mapsto \tfrac{1}{y_{21} y_{22}},
y_{23} \mapsto \zeta y_{21} y_{22} y_{23},\\
&y_{31} \mapsto y_{32},
y_{32} \mapsto \tfrac{1}{y_{31} y_{32}}, y_{33} \mapsto \zeta^2 y_{31} y_{32} y_{33},\\
f_3 : \ &
y_{11} \mapsto \zeta^2 y_{11}, y_{12} \mapsto \zeta^2 y_{12}, y_{13} \mapsto \zeta^2 y_{13},
y_{21} \mapsto \zeta^2 y_{21},  y_{22} \mapsto \zeta^2 y_{22}, y_{23} \mapsto \zeta^2 y_{23},\\
&y_{31} \mapsto \zeta^2 y_{31}, y_{32} \mapsto \zeta^2 y_{32},  y_{33} \mapsto \zeta^2 y_{33},\\
f_4 : \ &
y_{11} \mapsto y_{11}, y_{12} \mapsto y_{12}, y_{13} \mapsto y_{13},
y_{21} \mapsto y_{21}, y_{22} \mapsto y_{22},  y_{23} \mapsto \zeta y_{23},\\
&y_{31} \mapsto y_{31}, y_{32} \mapsto y_{32}, y_{33} \mapsto \zeta^2 y_{33},\\
f_5 : \ &
y_{11} \mapsto y_{11}, y_{12} \mapsto y_{12}, y_{13} \mapsto \zeta y_{13},
y_{21} \mapsto y_{21}, y_{22} \mapsto y_{22},  y_{23} \mapsto \zeta y_{23},\\
&y_{31} \mapsto y_{31}, y_{32} \mapsto y_{32}, y_{33} \mapsto \zeta y_{33}.
\end{align*}

Define
\begin{align*}
&z_1=\frac{y_{22}}{y_{32}},
z_2=\frac{y_{32}}{y_{12}},
z_3=\frac{y_{31}y_{32}}{y_{21}y_{22}},
z_4=\frac{y_{11}y_{12}}{y_{31}y_{32}},
z_5=y_{11}y_{22}y_{31},\\
&z_6=\frac{y_{12}y_{32}}{y_{21}y_{22}},
z_7=\frac{y_{13}y_{23}}{y_{33}^2},
z_8=\frac{y_{23}y_{33}}{y_{13}^2},
z_9=y_{13}y_{22}y_{23}y_{31}y_{32}y_{33}.
\end{align*}
Then
\begin{align*}
k(y_{11},y_{12},y_{13},y_{21},y_{22},y_{23},y_{31},y_{32},y_{33})^{\langle f_3,f_4,f_5\rangle}
=k(z_1,z_2,z_3,z_4,z_5,z_6,z_7,z_8,z_9)
\end{align*}
because the $z_i$'s are $\langle f_3,f_4,f_5\rangle$-invariants and
the determinant of the matrix of exponents is $-27$:
\begin{align}
{\rm Det}\left(
\begin{array}{ccccccccc}
 0 & 0 & 0 & 1 & 1 & 0 & 0 & 0 & 0 \\
 0 & -1 & 0 & 1 & 0 & 1 & 0 & 0 & 0 \\
 0 & 0 & 0 & 0 & 0 & 0 & 1 & -2 & 1 \\
 0 & 0 & -1 & 0 & 0 & -1 & 0 & 0 & 0 \\
 1 & 0 & -1 & 0 & 1 & -1 & 0 & 0 & 1 \\
 0 & 0 & 0 & 0 & 0 & 0 & 1 & 1 & 1 \\
 0 & 0 & 1 & -1 & 1 & 0 & 0 & 0 & 1 \\
 -1 & 1 & 1 & -1 & 0 & 1 & 0 & 0 & 1 \\
 0 & 0 & 0 & 0 & 0 & 0 & -2 & 1 & 1
\end{array}
\right) =-27.\label{det27}
\end{align}
The actions of $f_1$ and $f_2$
on $k(z_1,z_2,z_3,z_4,z_5,z_6,z_7,z_8,z_9)$ are given by
\begin{align*}
f_1 : \ & z_1 \mapsto z_2,
z_2 \mapsto \frac{1}{z_1 z_2},
z_3 \mapsto z_4,
z_4 \mapsto \frac{1}{z_3 z_4},\\
& z_5 \mapsto \frac{z_5}{z_1^2 z_3},
z_6 \mapsto \frac{z_1 z_6}{z_3},
z_7 \mapsto z_8,
z_8 \mapsto \frac{1}{z_7 z_8},
z_9 \mapsto \frac{z_4 z_9}{z_1},\\
f_2 : \ & z_1 \mapsto z_3,
z_2 \mapsto z_4,
z_3 \mapsto \frac{1}{z_1 z_3},
z_4 \mapsto \frac{1}{z_2 z_4},\\
& z_5 \mapsto z_6,
z_6 \mapsto \frac{1}{z_5 z_6},
z_7 \mapsto \frac{
 z_4 z_7}{z_3},
z_8 \mapsto \frac{
z_8}{z_3 z_4^2},
z_9 \mapsto \frac{
z_4 z_9}{z_1}.\\
\end{align*}
It remains to show that
$k(z_1,z_2,z_3,z_4,z_5,z_6,z_7,z_8,z_9)^{\langle f_1,f_2\rangle}$
is $k$-rational.

\medskip
Step 3.

Define two elements $a=\frac{1}{z_1z_3}$ and $b=\frac{1}{z_3}$.
They satisfy the following identities:
\begin{align*}
\left(a,f_1(a),f_1^2(a),f_2(a),f_2^2(a),\frac{a}{f_2(a)},
\frac{f_2(a)}{f_2^2(a)}\right)&=
\left(\frac{1}{z_1z_3},\frac{1}{z_2z_4},z_1z_2z_3z_4,
z_1,z_3,\frac{1}{z_1^2z_3},\frac{z_1}{z_3}\right),\\
\left(b,f_2(b),f_2^2(b),f_1(b),f_1^2(b),\frac{b}{f_1(b)},
\frac{f_1(b)}{f_1^2(b)}\right)&=
\left(\frac{1}{z_3},z_1z_3,\frac{1}{z_1},
\frac{1}{z_4},z_3z_4,\frac{z_4}{z_3}, \frac{1}{z_3z_4^2}\right).
\end{align*}

Apply Lemma \ref{lem4} twice to
$k(z_1,z_2,z_3,z_4,z_9)(z_5,z_6,z_7,z_8)$, there exist elements
$Z_5,Z_6,Z_7,Z_8$ $\in$ $k(z_1,z_2,z_3,z_4,z_9)$ such that
\[
k(z_1,z_2,z_3,z_4,z_9)(z_5,z_6,z_7,z_8)=k(z_1,z_2,z_3,z_4,z_9)(Z_5,Z_6,Z_7,Z_8)
\]
and $\sigma(Z_i)=Z_i$ for $5\leq i\leq 8$ and any $\sigma\in G$.

Namely, the action of $G$ on $k(z_1,z_2,z_3,z_4,Z_5,Z_6,Z_7,Z_8,z_9)$
is given by
\begin{align}
f_1 : \ & z_1 \mapsto z_2,
z_2 \mapsto \frac{1}{z_1 z_2},
z_3 \mapsto z_4,
z_4 \mapsto \frac{1}{z_3 z_4},\nonumber\\
& Z_5 \mapsto Z_5, Z_6 \mapsto Z_6, Z_7 \mapsto Z_7,
Z_8 \mapsto Z_8, z_9 \mapsto \frac{z_4 z_9}{z_1},\label{actz}\\
f_2 : \ & z_1 \mapsto z_3,
z_2 \mapsto z_4,
z_3 \mapsto \frac{1}{z_1 z_3},
z_4 \mapsto \frac{1}{z_2 z_4},\nonumber\\
& Z_5 \mapsto Z_5, Z_6 \mapsto Z_6, Z_7 \mapsto Z_7,
Z_8 \mapsto Z_8, z_9 \mapsto \frac{z_4 z_9}{z_1}.\nonumber
\end{align}
Hence it suffices to show that $k(z_1,z_2,z_3,z_4,z_9)^{\langle f_1,f_2\rangle}$
is $k$-rational.

\medskip
Step 4.

Define $w_1=\frac{z_4z_9}{z_1}, w_2=\frac{z_9}{z_1z_2z_3},
w_3=z_3, w_4=z_4, w_5=z_9$.
Then $k(z_1,z_2,z_3,z_4,z_9)=k(w_1,w_2,w_3,w_4,w_5)$ and
\begin{align*}
f_1 : \ & w_1 \mapsto w_2,
w_2 \mapsto w_5,
w_3 \mapsto w_4,
w_4 \mapsto \frac{1}{w_3 w_4},
w_5 \mapsto w_1,\\
f_2 : \ & w_1 \mapsto w_2,
w_2 \mapsto w_5,
w_3 \mapsto \frac{w_1}{w_3 w_4 w_5},
w_4 \mapsto \frac{w_2w_3}{w_1},
w_5 \mapsto w_1.
\end{align*}
Define
\begin{align*}
s_1&=w_1+w_2+w_5, s_2=w_1+\zeta^2w_2+\zeta w_5,
s_3=w_1+\zeta w_2+\zeta^2 w_5,\\
s_4&=\frac{1+\zeta^2 w_3+\zeta w_3w_4}{1+w_3+w_3w_4},
s_5=\frac{1+\zeta w_3+\zeta^2 w_3w_4}{1+w_3+w_3w_4}.
\end{align*}
Then $k(w_1,w_2,w_3,w_4,w_5)=k(s_1,s_2,s_3,s_4,s_5)$ and
\begin{align*}
f_1 : \ & s_1 \mapsto s_1,
s_2 \mapsto \zeta s_2,
s_3 \mapsto \zeta^2 s_3,
s_4 \mapsto \zeta s_4,
s_5 \mapsto \zeta^2 s_5,\\
f_2 : \ & s_1 \mapsto s_1,
s_2 \mapsto \zeta s_2,
s_3 \mapsto \zeta^2 s_3,
s_4 \mapsto \frac{\zeta^2(s_2+s_1s_4+s_3s_5)}{s_1+s_3s_4+s_2 s_5},
s_5 \mapsto \frac{\zeta(s_3+s_2s_4+s_1s_5)}{s_1+s_3s_4+s_2s_5}.
\end{align*}
Define $t_1=s_1, t_2=s_2^3, t_3=\frac{s_3}{s_2^2}, t_4=s_2^2s_4, t_5=s_2s_5$.
Then $k(s_1,s_2,s_3,s_4,s_5)^{\langle f_1\rangle}=k(t_1,t_2,t_3,t_4,t_5)$ and
\begin{align}
f_2 : \ & t_1 \mapsto t_1,
t_2 \mapsto t_2,
t_3 \mapsto t_3,
t_4 \mapsto \frac{\zeta(t_2+t_1t_4+t_2t_3t_5)}{t_1+t_3t_4+t_5},
t_5 \mapsto \frac{\zeta^2(t_2t_3+t_4+t_1t_5)}{t_1+t_3t_4+t_5}.\label{actt}
\end{align}
Hence we will show that
$k(t_1,t_2,t_3,t_4,t_5)^{\langle f_2\rangle}$ is $k$-rational.

\medskip
Step 5.

We use Theorem \ref{t2.6} to simplify the action of $f_2$.
Define $L_i\in k(t_1,t_2,t_3)[t_4,t_5]$ to be the polynomials
satisfying $f_2(t_4)=\frac{L_1}{L_0}$ and $f_2(t_5)=\frac{L_2}{L_0}$
in the above Formula (\ref{actt}).
The coefficient matrix of $L_0, L_1, L_2$ with respect to $t_4$, $t_5$ is
\begin{align*}
\left[
\begin{array}{ccc}
 t_1 & t_3 & 1 \\
 \zeta t_2 & \zeta t_1 & \zeta t_2t_3 \\
 \zeta^2 t_2 t_3 & \zeta^2 & \zeta^2 t_1
\end{array}
\right]
\end{align*}
whose characteristic polynomial is $T^3-D$ where
$D=t_1^3 + t_2 - 3 t_1 t_2 t_3 + t_2^2 t_3^3$.
By Theorem \ref{t2.6}, there exist $u_4,u_5$ such
that $k(t_1,t_2,t_3,u_4,u_5)=k(t_1,t_2,t_3,t_4,t_5)$ and
\begin{align*}
f_2 : \ & t_1 \mapsto t_1, t_2 \mapsto t_2, t_3 \mapsto t_3, u_4
\mapsto u_5, u_5 \mapsto \frac{t_1^3 + t_2 - 3 t_1 t_2 t_3 + t_2^2
t_3^3}{u_4u_5}.
\end{align*}
Define $U_1=\frac{t_1t_3-1}{t_3}, U_2=\frac{t_2t_3^3-1}{t_3},
U_3=t_3, U_4=u_4, U_5=u_5$.
Then $k(t_1,t_2,t_3,u_4,u_5)=k(U_1,U_2,U_3,U_4,U_5)$ and
\begin{align*}
f_2 : \ & U_1 \mapsto U_1,
U_2 \mapsto U_2,
U_3 \mapsto U_3,
U_4 \mapsto U_5,
U_5 \mapsto \frac{3U_1^2-3U_1U_2+U_2^2+U_1^3U_3}{U_3U_4U_5}.
\end{align*}
Note that both the numerator and the denominator of $f_2(U_5)$ are
linear in $U_3$.
Define $v_1=U_1, v_2=U_2,
v_3=\frac{3U_1^2-3U_1U_2+U_2^2+U_1^3U_3}{U_3U_4U_5},
v_4=U_4, v_5=U_5$.
Then $k(U_1,U_2,U_3,U_4,U_5)$ $=$ $k(v_1,v_2,v_3,v_4,v_5)$ and
\begin{align*}
f_2 : \ & v_1 \mapsto v_1,
v_2 \mapsto v_2,
v_3 \mapsto v_4,
v_4 \mapsto v_5,
v_5 \mapsto v_3.
\end{align*}

The cyclic group $\langle f_2\rangle$ acts linearly on
$k(v_1,v_2,v_3,v_4,v_5)$.
It follows from Theorem \ref{t2.2} that
$k(v_1,v_2,v_3,v_4,v_5)^{\langle f_2\rangle}$ is $k$-rational.
Hence $k(V_{65})^G$ is $k$-rational.\\
\end{idef}

\bigskip
\begin{idef}{Case 2} $G=G(66)$.

\medskip
For $G=G(66)$, we can follow the same way as in Case 1: $G=G(65)$.
In the present situation, the formula (\ref{actz}) should be
replaced by the following
\begin{align*}
f_1 : \ & z_1 \mapsto z_2, z_2 \mapsto \frac{1}{z_1 z_2}, z_3
\mapsto z_4, z_4 \mapsto \frac{1}{z_3 z_4},
z_9 \mapsto \zeta \frac{z_4 z_9}{z_1},\\
f_2 : \ & z_1 \mapsto z_3, z_2 \mapsto z_4, z_3 \mapsto
\frac{1}{z_1 z_3}, z_4 \mapsto \frac{1}{z_2 z_4}, z_9 \mapsto
\frac{z_4 z_9}{z_1}.
\end{align*}
Apply Theorem \ref{t2.3} to $k(z_1,z_2,z_3,z_4)(z_9)$. We find an
element $Z_9$ with the property that
$k(z_1,z_2,z_3,z_4,z_9)=k(z_1,z_2,z_3,z_4,Z_9)$ and
$f_1(Z_9)=f_2(Z_9)=Z_9$. Thus $k(G(66))$ is $k$-isomorphic to
$k(G(65))$. Hence the result.
\end{idef}

{}From the above proof (see Step 3 of Case 1, in particular), we
obtain the following proposition as a corollary.
\begin{prop}\label{cor1}
Let $k$ be any field.
Let $\langle f_1,f_2\rangle\simeq C_3\times C_3$ act on
the rational function field $k(z_1,z_2,z_3,z_4,z_9)$ with
five variables $z_1,z_2,z_3,z_4,z_9$ over $k$ by $k$-automorphism
\begin{align*}
f_1 : \ & z_1 \mapsto z_2,
z_2 \mapsto \frac{1}{z_1 z_2},
z_3 \mapsto z_4,
z_4 \mapsto \frac{1}{z_3 z_4},
z_9 \mapsto z_9,\\
f_2 : \ & z_1 \mapsto z_3,
z_2 \mapsto z_4,
z_3 \mapsto \frac{1}{z_1 z_3},
z_4 \mapsto \frac{1}{z_2 z_4},
z_9 \mapsto z_9.
\end{align*}
Then $k(z_1,z_2,z_3,z_4,z_9)^{\langle f_1,f_2\rangle}$ is
$k$-rational. However, it is not clear to us
whether $k(z_1,z_2,z_3,z_4)^{\langle f_1,f_2\rangle}$ is
$k$-rational or not.
\end{prop}

%

\section{Proof of Theorem \ref{thmain1}}

Proof of Theorem \ref{thmain1} ---------------------

 Let $G$ be a group of order $243$.

If $\fn{char} k =3$, then $k(G)$ is $k$-rational by Theorem
\ref{t2.1}.
{}From now on, we will assume that $\fn{char} k \neq 3$
and $k$ contains $\zeta_e$ where $e=\exp(G)$.

By Theorem \ref{t1.9},
$\fn{Br}_{v,\bm{C}}(\bm{C}(G))\neq 0$ if and only if
$G$ belongs to $\Phi_{10}$.
Hence we should consider the cases $\Phi_1,\ldots,\Phi_9$.

If $G$ belongs to $\Phi_1$, then $G$ is abelian group and
hence $k(G)$ is $k$-rational by Theorem \ref{t2.2}.

If $G$ belongs to $\Phi_2$, $\Phi_3$, $\Phi_4$ or $\Phi_9$,
then there exists a normal abelian subgroup $N$ of $G$
such that $G/N$ is cyclic of order $3$ (these groups correspond to the
groups in Bender's classification \cite[Section 4]{Be}). Hence
$k(G)$ is $k$-rational by Theorem \ref{t2.5}.

If $G$ belongs to $\Phi_8$, then there exists a normal cyclic subgroup $C_{27}$
of $G$ of order $27$ such that $G/C_{27}$ is cyclic of order $9$.
Hence $k(G)$ is $k$-rational by Theorem \ref{tKa2}.

If $G$ belongs to $\Phi_6$,
$k(G)$ is $k$-rational by a result as in Section 4.

If $G$ belongs to $\Phi_5$,
$k(G)$ is $k$-rational by a result as in Section 5.
\qed

\bigskip
Proof of Proposition \ref{p1.15} --------------------

Let $G$ be a group of order $243$ which belongs to $\Phi_{7}$,
i.e. $G=G(i)$, $(56\leq i\leq 60)$.

\bigskip
\begin{idef}{Case 1} $G=G(56)$ and $G(60)$.

\medskip
We choose the representation $G \longrightarrow GL(V_{56})$
and $G\longrightarrow GL(V_{60})$ given as in Section 3.

By Theorem \ref{t2.3}, $k(G(56))=k(V_{56})^{G(56)}(u_i : 1 \le i
\le 234)$ and $k(G(60))=K(V_{60})(v_i : 1 \le i \le 234)$ for some
algebraic independent variables $u_i, v_i$, $1 \le i \le 234$.

In $k(V_{56})$, define $X_i=x_i/x_1$ for $2 \le i \le 9$. Then
$k(x_1, \ldots, x_9)^{G(56)}=k(X_i:2 \le i \le 9)^{G(56)}(u)$ by
Theorem \ref{t2.4}. For $k(V_{60})=k(x_1, \ldots, x_9)$, also
define $X_i=x_i/x_1$ for $2 \le i \le 9$.
Then $k(x_1, \ldots, x_9)^{G(60)}=k(X_i:2 \le i \le 9)^{G(60)}(v)$
by Theorem \ref{t2.4}.

Compare the action of $G(56)$ on $k(X_i:2 \le i \le 9)$ with that
of $G(60)$ on $k(X_i:2 \le i \le 9)$. We find that they are
completely the same. Hence $k(X_i:2 \le i \le 9)^{G(56)}$ is
$k$-isomorphic to $k(X_i:2 \le i \le 9)^{G(60)}$.
Thus $k(G(56))$ is $k$-isomorphic to $k(G(60))$.
\end{idef}

\bigskip
\begin{idef}{Case 2} $G(i)$, $56\leq i\leq 59$.

\medskip
In these cases, the idea is the same as in Case 1 of Section 5.

For $56\leq i\leq 59$, choose the faithful representation
$G \rightarrow GL(V_i)$ according to the matrices as in Section 3.
We take $k(V_i)=k(x_{11},x_{12},x_{13},x_{21},x_{22},x_{23},x_{31},x_{32},x_{33})$
and define the same
$y_{11}=\tfrac{x_{11}}{x_{12}}$, $y_{12}=\tfrac{x_{12}}{x_{13}}$,
$y_{13}=x_{13}$, $y_{21}=\tfrac{x_{21}}{x_{22}}$, $y_{22}=\tfrac{x_{22}}{x_{23}}$,
$y_{23}=x_{23}$ as in Case 1 of Section 5.
Then we have
\[
k(x_{11},x_{12},x_{13},x_{21},x_{22},x_{23},x_{31},x_{32},x_{33})
=k(y_{11},y_{12},y_{13},y_{21},y_{22},y_{23},y_{31},y_{32},y_{33}).
\]
Define
\begin{align*}
&z_1=\frac{y_{22}}{y_{32}},
z_2=\frac{y_{32}}{y_{12}},
z_3=\frac{y_{31}y_{32}}{y_{21}y_{22}},
z_4=\frac{y_{11}y_{12}}{y_{31}y_{32}},
z_5=y_{11}y_{22}y_{31},\\
&z_6=m_1 \frac{y_{12}y_{32}}{y_{21}y_{22}},
z_7=\frac{y_{13}y_{23}}{y_{33}^2},
z_8=m_2 \frac{y_{23}y_{33}}{y_{13}^2},
z_9=y_{13}y_{22}y_{23}y_{31}y_{32}y_{33}
\end{align*}
where
\begin{align*}
(m_1,m_2)=
\begin{cases}
(1,1) &{\rm if}\quad i=56,\\
(\zeta^2,1) &{\rm if}\quad i=57,\\
(\zeta,1) &{\rm if}\quad i=58,\\
(\zeta,\zeta) &{\rm if}\quad i=59.
\end{cases}
\end{align*}
By evaluating the determinant of the exponents (see Formula
(\ref{det27})), we have
\begin{align*}
k(y_{11},y_{12},y_{13},y_{21},y_{22},y_{23},y_{31},y_{32},y_{33})^{\langle f_3,f_4,f_5\rangle}
=k(z_1,z_2,z_3,z_4,z_5,z_6,z_7,z_8,z_9)
\end{align*}
and the actions of $G(i)$, $(56\leq i\leq 59)$
on $k(z_1,z_2,z_3,z_4,z_5,z_6,z_7,z_8,z_9)$ are given by
\begin{align*}
f_1 : \ & z_1 \mapsto z_2,
z_2 \mapsto \frac{1}{z_1 z_2},
z_3 \mapsto z_4,
z_4 \mapsto \frac{1}{z_3 z_4},\\
& z_5 \mapsto \frac{z_5}{z_1^2 z_3},
z_6 \mapsto \frac{z_1 z_6}{z_3},
z_7 \mapsto z_8,
z_8 \mapsto \frac{1}{z_7 z_8},
z_9 \mapsto m_2 \frac{z_4 z_9}{z_1},\\
f_2 : \ & z_1 \mapsto z_3,
z_2 \mapsto z_4,
z_3 \mapsto \frac{1}{z_1 z_3},
z_4 \mapsto \frac{1}{z_2 z_4},\\
& z_5 \mapsto z_6,
z_6 \mapsto \frac{1}{z_5 z_6},
z_7 \mapsto \zeta \frac{z_4 z_7}{z_3},
z_8 \mapsto \zeta \frac{z_8}{z_3 z_4^2},
z_9 \mapsto m_1^2 \zeta \frac{z_4 z_9}{z_1}.
\end{align*}

Applying Theorem \ref{t2.3} to
$k(z_1,z_2,z_3,z_4,z_5,z_6,z_7,z_8)(z_9)$, there exists
$G(i)$-invariant $Z_9$ such that
$k(z_1,z_2,z_3,z_4,z_5,z_6,z_7,z_8,z_9)^{G(i)}
=k(z_1,z_2,z_3,z_4,z_5,z_6,z_7,z_8)^{G(i)}(Z_9)$. Note that the
actions of $G(i)$ $(56\leq i\leq 59)$ on these
$k(z_1,z_2,z_3,z_4,z_5,z_6,z_7,z_8)$ are exactly the same. Hence
the result.
\end{idef}

\bigskip
Proof of Proposition \ref{p1.16} --------------------

Let $G$ be a group of order $243$ which belongs to $\Phi_{10}$,
i.e. $G=G(i)$, $(28\leq i\leq 30)$. For $i=28, 29, 30$, we choose
the representation $G(i)\longrightarrow GL(V_i)$ given as in
Section 3.

The action of $G(28)$ on $k(V_{28})=
k(x_{11},x_{12},x_{13},x_{21},x_{22},x_{23},x_{31},x_{32},x_{33})$
is given by
\begin{align*}
f_1 : \ &
x_{11}\mapsto x_{12},x_{12}\mapsto x_{13},x_{13}\mapsto x_{11},
x_{21}\mapsto x_{22},x_{22}\mapsto x_{23},x_{23}\mapsto x_{21},\\
&x_{31}\mapsto x_{32},x_{32}\mapsto x_{33},x_{33}\mapsto x_{31},\\
f_2 : \ &
x_{11}\mapsto x_{31}, x_{12}\mapsto \eta^8 x_{32},
x_{13}\mapsto \eta^7 x_{33},
x_{21}\mapsto x_{11}, x_{22}\mapsto \eta^5 x_{12},
x_{23}\mapsto \eta x_{13},\\
&x_{31}\mapsto x_{21}, x_{32}\mapsto \eta^2 x_{22},
x_{33}\mapsto \eta^4 x_{23},\\
f_3 : \ &
x_{11}\mapsto \eta^8 x_{11},x_{12}\mapsto \eta^8 x_{12},
x_{13}\mapsto \eta^2 x_{13},
x_{21}\mapsto \eta^5 x_{21},x_{22}\mapsto \eta^5 x_{22},
x_{23}\mapsto \eta^8 x_{23},\\
&x_{31}\mapsto \eta^2 x_{31}, x_{32}\mapsto \eta^2 x_{32},
x_{33}\mapsto \eta^5 x_{33},\\
f_4 : \ &
x_{11}\mapsto x_{11},x_{12}\mapsto \zeta x_{12},
x_{13}\mapsto \zeta^2 x_{13},
x_{21}\mapsto x_{21},x_{22}\mapsto \zeta x_{22},
x_{23}\mapsto \zeta^2 x_{23},\\
&x_{31}\mapsto  x_{31},x_{32}\mapsto \zeta x_{32},
x_{33}\mapsto \zeta^2 x_{33},\\
f_5 : \ &
x_{11}\mapsto \zeta x_{11},x_{12}\mapsto \zeta x_{12},
x_{13}\mapsto \zeta x_{13},
x_{21}\mapsto \zeta x_{21},x_{22}\mapsto \zeta x_{22},
x_{23}\mapsto \zeta x_{23},\\
&x_{31}\mapsto \zeta x_{31},x_{32}\mapsto \zeta x_{32},
x_{33}\mapsto \zeta x_{33}.
\end{align*}
Define $y_{11}=\frac{x_{11}}{x_{12}}$, $y_{12}=\frac{x_{12}}{x_{13}}$,
$y_{13}=x_{13}$, $y_{21}=\frac{x_{21}}{x_{22}}$, $y_{22}=\frac{x_{22}}{x_{23}}$,
$y_{23}=x_{23}$, $y_{31}=\frac{x_{31}}{x_{32}}$, $y_{32}=\frac{x_{32}}{x_{33}}$,
$y_{33}=x_{33}$.
Then
\[
k(x_{11},x_{12},x_{13},x_{21},x_{22},x_{23},x_{31},x_{32},x_{33})
=k(y_{11},y_{12},y_{13},x_{21},y_{22},y_{23},y_{31},y_{32},y_{33})
\]
and
\begin{align*}
f_1 : \ &
y_{11} \mapsto y_{12}, y_{12} \mapsto \frac{1}{y_{11}y_{12}},
y_{13} \mapsto y_{11}y_{12}y_{13},
y_{21} \mapsto y_{22}, y_{22} \mapsto \frac{1}{y_{21}y_{22}},
y_{23} \mapsto y_{21}y_{22}y_{23},\\
&y_{31} \mapsto y_{32}, y_{32} \mapsto \frac{1}{y_{31}y_{32}},
y_{33} \mapsto y_{31}y_{32}y_{33},\\
f_2 : \ &
y_{11} \mapsto \eta y_{31}, y_{12} \mapsto \eta y_{32},
y_{13} \mapsto \eta^2 y_{33},
y_{21} \mapsto \eta^4 y_{11},
y_{22} \mapsto \eta^4 y_{12},
y_{23} \mapsto \eta y_{13},\\
&y_{31} \mapsto \eta^7 y_{21},
y_{32} \mapsto \eta^7 y_{22},
y_{33} \mapsto \eta^4 y_{23},\\
f_3 : \ &
y_{11} \mapsto y_{11}, y_{12} \mapsto \zeta^2 y_{12}, y_{13} \mapsto \eta^2 y_{13},
y_{21} \mapsto y_{21},  y_{22} \mapsto \zeta^2 y_{22}, y_{23} \mapsto \eta^8 y_{23},\\
&y_{31} \mapsto y_{31}, y_{32} \mapsto \zeta^2 y_{32},  y_{33} \mapsto \eta^5 y_{33},\\
f_4 : \ &
y_{11} \mapsto \zeta^2 y_{11}, y_{12} \mapsto \zeta^2 y_{12}, y_{13} \mapsto \zeta^2 y_{13},
y_{21} \mapsto \zeta^2 y_{21}, y_{22} \mapsto \zeta^2 y_{22},  y_{23} \mapsto \zeta^2 y_{23},\\
&y_{31} \mapsto \zeta^2 y_{31}, y_{32} \mapsto \zeta^2 y_{32}, y_{33} \mapsto \zeta^2 y_{33},\\
f_5 : \ &
y_{11} \mapsto y_{11}, y_{12} \mapsto y_{12}, y_{13} \mapsto \zeta y_{13},
y_{21} \mapsto y_{21}, y_{22} \mapsto y_{22},  y_{23} \mapsto \zeta y_{23},\\
&y_{31} \mapsto y_{31}, y_{32} \mapsto y_{32}, y_{33} \mapsto \zeta y_{33}.
\end{align*}

For $G=G(29)$ and $G(30)$,
we follow the same way to $G(28)$ and
take the same $y_{ij}$'s.
Define
\begin{align*}
&z_1=\frac{y_{12}}{y_{22}},
z_2=\frac{y_{21} y_{22}}{y_{11} y_{12}},
z_3=\frac{y_{32}}{y_{12} \zeta},
z_4=\frac{y_{11} y_{12} \zeta^2}{y_{31} y_{32}},
z_5=\frac{y_{12} y_{13} y_{21} y_{23}\zeta^2}{y_{31} y_{32} y_{33}^2},\\
&z_6=\frac{y_{11} y_{13} y_{32} y_{33} \zeta^2}{y_{21} y_{22} y_{23}^2},
z_7=\frac{y_{11} y_{22} y_{32} y_{33}}{y_{23}},
z_8=m_1 \frac{y_{12} y_{33}}{y_{21}^2 y_{22}^2 y_{23}},
z_9=\frac{1}{y_{11} y_{12} y_{13} y_{22} y_{23} y_{33}}
\end{align*}
where
\begin{align*}
m_1=
\begin{cases}
1 &{\rm if}\quad i=28,\\
\zeta &{\rm if}\quad i=29,\\
\zeta^2 &{\rm if}\quad i=30.
\end{cases}
\end{align*}
Then we have
\begin{align*}
k(y_{11},y_{12},y_{13},y_{21},y_{22},y_{23},y_{31},y_{32},y_{33})^{\langle f_3,f_4,f_5\rangle}
=k(z_1,z_2,z_3,z_4,z_5,z_6,z_7,z_8,z_9)
\end{align*}
because the $z_i$'s are fixed by the actions of $f_3,f_4,f_5$ and
the determinant of the matrix of exponents is $27$:
\begin{align*}
{\rm Det}\left(
\begin{array}{ccccccccc}
 0 & -1 & 0 & 1 & 0 & 1 & 1 & 0 & -1 \\
 1 & -1 & -1 & 1 & 1 & 0 & 0 & 1 & -1 \\
 0 & 0 & 0 & 0 & 1 & 1 & 0 & 0 & -1 \\
 0 & 1 & 0 & 0 & 1 & -1 & 0 & -2 & 0 \\
 -1 & 1 & 0 & 0 & 0 & -1 & 1 & -2 & -1 \\
 0 & 0 & 0 & 0 & 1 & -2 & -1 & -1 & -1 \\
 0 & 0 & 0 & -1 & -1 & 0 & 0 & 0 & 0 \\
 0 & 0 & 1 & -1 & -1 & 1 & 1 & 0 & 0 \\
 0 & 0 & 0 & 0 & -2 & 1 & 1 & 1 & -1 \\
\end{array}
\right) =27.
\end{align*}
The actions of $G(28)$, $G(29)$ and $G(30)$
on $k(z_1,z_2,z_3,z_4,z_5,z_6,z_7,z_8,z_9)$ are
given by
\begin{align*}
f_1 : \ & z_1 \mapsto z_2,
z_2 \mapsto \frac{1}{z_1 z_2},
z_3 \mapsto z_4,
z_4 \mapsto \frac{1}{z_3 z_4},\\
& z_5 \mapsto \zeta^2 \frac{z_5}{z_1^2z_3},
z_6 \mapsto \zeta^2 \frac{z_1 z_6}{z_3},
z_7 \mapsto z_8,
z_8 \mapsto \frac{z_1^2 z_2 z_6}{z_5z_7z_8},
z_9 \mapsto \zeta m_1^2 \frac{z_4 z_9}{z_1},\\
f_2 : \ & z_1 \mapsto z_3,
z_2 \mapsto z_4,
z_3 \mapsto \frac{1}{z_1 z_3},
z_4 \mapsto \frac{1}{z_2 z_4},\\
& z_5 \mapsto z_6,
z_6 \mapsto \frac{1}{z_5z_6},
z_7 \mapsto \zeta^2 \frac{z_7}{z_2z_3z_4z_6},
z_8 \mapsto \frac{z_2z_3^2z_8}{z_6},
z_9 \mapsto \zeta \frac{z_4z_9}{z_1}.
\end{align*}

%

By applying Theorem \ref{t2.4} to
$k(z_1,z_2,z_3,z_4,z_5,z_6,z_7,z_8)(z_9)$, we reduce the question
on the rationality of $k(z_1,z_2,z_3,z_4,z_5,z_6,z_7,z_8)^{\langle
f_1,f_2\rangle}$. But the actions of $f_1, f_2$ on
$k(z_1,z_2,z_3,z_4,z_5,z_6,z_7,z_8)$ are the same for these three
groups. Thus we finish the proof.

\newpage
\renewcommand{\refname}{\centering{References}}

\end{document}